\documentclass[10pt]{article}
\usepackage{amsmath,amssymb, amsthm} 
\usepackage[varg]{pxfonts}   
\usepackage{bbm}             
\usepackage{nicefrac}        
\usepackage[all]{xy}         
\usepackage[dvips]{graphicx}
\usepackage{psfrag}          

\theoremstyle{plain}
    \newtheorem{thm}{Theorem}
    
    \newtheorem{corol}[thm]{Corollary}
    \newtheorem{prop}{Proposition}[section]
    \newtheorem{othertheorem}[prop]{Theorem}
    \newtheorem{lemma}[prop]{Lemma}
    \newtheorem{subl}[prop]{Sublemma}
    
    \newtheorem*{main}{Main Lemma}
\theoremstyle{definition}
    
\theoremstyle{remark}
    \newtheorem{rem}[prop]{Remark}

    \newtheorem*{ack}{Acknowledgements}

\numberwithin{equation}{section}

\def\bysame{\leavevmode\hbox to3em{\hrulefill}\thinspace} 
\newcommand{\R}{\mathbb{R}}\newcommand{\Z}{\mathbb{Z}}\newcommand{\N}{\mathbb{N}}

\newcommand{\D}{\mathbb{D}}
\renewcommand{\P}{\mathbb{P}}
\newcommand{\cA}{\mathcal{A}}\newcommand{\cB}{\mathcal{B}}\newcommand{\cC}{\mathcal{C}}

\newcommand{\cJ}{\mathcal{J}}\newcommand{\cL}{\mathcal{L}}
\newcommand{\cN}{\mathcal{N}}
\newcommand{\cR}{\mathcal{R}}
\newcommand{\cV}{\mathcal{V}}

\newcommand{\eps}{\varepsilon}

\renewcommand{\setminus}{\smallsetminus}

\newcommand{\LE}{\mathit{LE}}

\newcommand{\Id}{\mathit{Id}}
\newcommand{\Diff}{\mathrm{Diff}}
\newcommand{\PH}{\mathit{PH}}

\renewcommand{\angle}{\measuredangle}
\newcommand{\m}{\mathbf{m}}
\newcommand{\vetor}[1]{\tfrac{\partial}{\partial {#1}}}
\newcommand{\KII}{K_{\mathrm{II}}}
\newcommand{\KIII}{K_{\mathrm{III}}}
\newcommand{\KIV}{K_{\mathrm{IV}}}
\renewcommand{\v}{\mathbf{v}}
\newcommand{\ivec}{{\vec{\imath}}}
\newcommand{\ccircs}{\mathbin{\mathord{\circ}\mathord{\circ}\mathord{\circ}}}
\newcommand{\nor}[1]{\mathopen{|{\kern-.05em}|{\kern-.05em}|}{#1}\mathclose{{|{\kern-.05em}|{\kern-.05em}|}}}
\newcommand{\bignor}[1]{\mathopen{{\big |}{\kern-.05em}{\big |}{\kern-.05em}{\big |}}{#1}\mathclose{{{\big |}{\kern-.05em}{\big |}{\kern-.05em}{\big |}}}} 
\newcommand{\circulo}{\nicefrac{\R}{\pi\Z}}

\begin{document}

\title
{$C^1$-Generic Symplectic Diffeomorphisms: Partial Hyperbolicity and Zero Center Lyapunov Exponents}

\author{Jairo Bochi\thanks{Partially supported by a CNPq--Brazil research grant.}}

\date{September 14, 2008}
\maketitle

\begin{abstract}
We prove that if $f$ is a $C^1$-generic symplectic diffeomorphism then
the Oseledets splitting along almost every orbit is either trivial or partially hyperbolic.
In addition, if $f$ is not Anosov then all the exponents in the center bundle vanish.
This establishes in full a result announced by R.~Ma\~{n}\'{e} in the ICM~1983.
The main technical novelty is a probabilistic method
for the construction of perturbations, using random walks.
\end{abstract}



\tableofcontents

\section{Introduction}

One of the cornerstones of differentiable ergodic theory is
the Theorem of Oseledets~\cite{Oseledets}.
Given a diffeomorphism $f \colon M \to M$ of a closed manifold $M$,
a point $x\in M$ is called \emph{regular} if
there exists a \emph{Oseledets (or Lyapunov) splitting}
$E^1(x) \oplus \cdots \oplus E^{k(x)}(x)$
of the tangent space $T_x M$,
and
corresponding \emph{Lyapunov exponents}
$\hat{\lambda}_1(x) > \cdots > \hat{\lambda}_{k(x)}(x)$, so that
\begin{equation}\label{e.def Lyapunov}
\lim_{n \to \pm \infty} \frac{1}{n}\log\|Df^n(x) \cdot v\| = \hat{\lambda}_j(x)
\quad \text{for all non-zero $v\in E^j(x)$.}
\end{equation}
(Here $\|\mathord{\cdot}\|$ is any Riemannian metric on $M$.)
The Theorem of Oseledets asserts that
regular points form a full probability subset $R$ of $M$
(meaning that $\nu(R)=1$ for any $f$-invariant probability measure~$\nu$).
Now, quoting Ma\~{n}\'{e}~\cite{Mane ICM},
\begin{quote}
Oseledets' theorem is essentially a measure theoretical result and therefore the information
it provides holds only in that category.
For instance, the Lyapunov splitting is just a measurable function of the point and the limits
defining the Lyapunov exponents are not uniform.
It is clear that this is not a deficiency of the theorem
but the natural counterweight to its remarkable generality.
However, one can pose the problem \dots \   
of whether these aspects can be substantially improved by working under generic
conditions.
\end{quote}
These words suggest that
a theory of generic dynamical systems must include improved versions of the Oseledets'  Theorem.
Indeed, the paper \cite{BV Annals} by Viana and the author establishes such a result
for the class of volume-preserving $C^1$-diffeomorphisms.

The present work obtains the $C^1$-generic improvement of the Oseledets' Theorem for the class
of \emph{symplectic diffeomorphisms}.
Our main result is precisely the strongest one stated and left open by Ma\~{n}\'{e} in 1983 \cite{Mane ICM}.

\medskip

Let $\Lambda \subset M$ be an invariant set
for a diffeomorphism $F \colon M \to M$ of a closed manifold.
A $Df$-invariant splitting
$T_\Lambda M = E^1 \oplus \cdots \oplus E^k$ into $k \ge 2$ non-zero bundles of constant dimensions
is called a \emph{dominated splitting} if
there is a constant $\tau>1$ such that,
up to a change\footnote{The usual definition
without change of the metric is explained in \S\ref{ss.review}.
Here we are using Gourmelon's \emph{adapted metric} \cite{Gourmelon}
to simplify the exposition.}
of the Riemannian metric on $M$,
\begin{equation}\label{e.def DS adapted}
\frac{\| Df(x) \cdot v_i \|}{\|v_i\|} > \tau \,
\frac{\| Df(x) \cdot v_j \|}{\|v_j\|}
\quad \text{for all $x \in \Lambda$, non-zero $v_i\in E^i(x)$, $v_j\in E^j(x)$ with $i<j$.}
\end{equation}
Dominated splittings enjoy strong properties:
they can be uniquely extended to the closure of $\Lambda$,
the spaces $E^i$ vary continuously,
and the angles between them are uniformly bounded away from zero.
Domination is also called \emph{projective hyperbolicity},
see~\cite{BV IHP}.

\medskip

From now on we assume that the closed manifold $M$ is \emph{symplectic},
that is, it supports a closed non-degenerate $2$-form $\omega$.
Let $2N$ be the dimension of $M$.
Let $\Diff^1_\omega(M)$ be the space of $\omega$-preserving $C^1$ diffeomorphisms,
endowed with the $C^1$ topology.
Let $\mu$ be the measure induced by the volume form $\omega^{\wedge N}$.
We assume that $\omega$ is normalized so that $\mu(M)=1$.
All the ``almost sure'' statements in the sequel refer to this measure.

Here is the generic improvement of the Oseledets' Theorem obtained in this paper:

\begin{thm}\label{t.main}
There exists a residual $\cR \subset \Diff_\omega^1(M)$ such that if $f \in \cR$
then for almost every point $x$, the Oseledets splitting $T_x M = E^1(x) \oplus \cdots \oplus E^{k(x)}(x)$
is either trivial  or dominated along the orbit of $x$.
\end{thm}

The first alternative means that $k(x)=1$, that is, all Lyapunov exponents at $x$ are zero.
In the second alternative, we can in fact obtain even sharper information,
using the general fact (proven in \cite{BV Cambridge})
that \emph{for symplectic maps, dominated splittings are automatically partially hyperbolic}.
Let us postpone the precise statement to \S\ref{ss.review},
and explain the consequences for the generic maps from Theorem~\ref{t.main}.

First, the Lyapunov exponents of any symplectic diffeomorphism are \emph{symmetric}:
if $\lambda$ is an exponent at the point $x$ then so is $-\lambda$,
and they have the same multiplicity.
(The multiplicity of the Lyapunov exponent $\hat{\lambda}_j(x)$ as in~\eqref{e.def Lyapunov}
is defined as $\dim E^j(x)$.)

From the Oseledets splitting at a regular point $x$,
we form the \emph{zipped Oseledets splitting}:
\begin{equation}\label{e.zip}
T_x M = E^+(x) \oplus E^0(x) \oplus E^-(x),
\end{equation}
where $E^+(x)$, $E^0(x)$, and $E^-(x)$ are the sums of
the spaces $E^j(x)$ corresponding to positive, zero, and negative $\hat{\lambda}_j(x)$, respectively.
By symplectic symmetry, $\dim E^+(x) = \dim E^-(x)$ and $\dim E^0(x)$ is even.

Assume that the point $x$ is such that
the full Oseledets splitting along the orbit of $x$ is dominated.
Then so is the zipped splitting $E^+\oplus E^0 \oplus E^-$.
Besides,
\emph{the space $E^+$ is uniformly expanding and the space $E^-$ is uniformly contracting}.
In other words, there is a constant $\sigma>1$ such that, up to a change of the Riemannian metric on $M$,
$$
\left.
\begin{array}{l}
\|Df(y) \cdot v_+\|  \ge \sigma \| v_+\| \\
\|Df(y) \cdot v_-\|  \le \sigma^{-1} \| v_-\|
\end{array}
\right\}
\text{ for all $y = f^n(x)$, $n \in \Z$, $v_+ \in E^+(y)$, $v_- \in E^-(y)$.}
$$
We say that the zipped Oseledets splitting is \emph{partially hyperbolic}.
It is evident that this is a much stronger conclusion
than just the asymptotic expansion/contraction provided by the bare Oseledets Theorem.

In the case that $E^0=\{0\}$, partial hyperbolicity becomes
the usual notion of \emph{uniform hyperbolicity}.
Another useful fact (also from \cite{BV Cambridge})
is that uniformly hyperbolic sets generically have either zero or full volume.
Thus (see \S\ref{ss.review} for full details)
we obtain the following complement of Theorem~\ref{t.main}:

\begin{corol}\label{c.main}
A $C^1$-generic symplectic diffeomorphism $f$
satisfies one and only one of the alternatives below:
\begin{enumerate}
\item $f$ is an Anosov diffeomorphism; that is,
there exists a uniformly hyperbolic splitting $TM = E^+ \oplus E^-$
that coincides with the zipped Oseledets splitting at a.~e.\ point.

\item For almost every point $x \in M$, either all Lyapunov exponents at $x$ are zero,
or the zipped Oseledets splitting $T_\Lambda M = E^+ \oplus E^0 \oplus E^-$ over the orbit $\Lambda$ of $x$
is partially hyperbolic with center dimension $\dim E^0$ at least $2$.
\end{enumerate}
\end{corol}

The statement of Corollary~\ref{c.main} is due to Ma\~{n}\'{e}, see \cite{Mane ICM}.
Its $2$-dimensional version, asserting that a generic area-preserving diffeomorphism either is Anosov
or has zero metric entropy, was established by the author in \cite{B GZLE}.
Some of the key ideas of the proof in \cite{B GZLE} came from the outline \cite{Mane sketch} left by Ma\~{n}\'{e}.
In \cite{BV Annals}, Viana and the author proved
a weaker version of Corollary~\ref{c.main} (without the partial hyperbolicity).
The paper \cite{BV Annals} also proves the full version of Theorem~\ref{t.main} for volume-preserving diffeomorphisms.
(The statement is word-by-word the same, only replacing the symplectic form $\omega$ by a volume form.)

There are results of similar nature for
volume-preserving and hamiltonian flows (currently only in low dimensions),
see \cite{Bessa}, \cite{Bessa LopesD},
and
for linear cocycles (deterministic products of matrices),
see \cite{BV Annals}, \cite{B Fayad}.

While this paper is the symplectic counterpart to \cite{BV Annals},
the present proofs required much more than technical adaptations.
To achieve our goal, we develop here a new perturbation method
that uses random walks.
See \S\ref{ss.preview} for an overview.
Other examples in the literature where probabilistic arguments are used 
to find dynamical systems with special properties are \cite{Moreira Yoccoz}, \cite[page~196]{Dolgo random}.

\medskip

Let us explore some consequences of the results above.
If $f$ is a generic non-Anosov map then the manifold is covered mod~$0$
by two disjoint invariant sets $Z$ and $D$ such that in $Z$ all exponents vanish,
and $D$ can be written as a non-decreasing union
$D = \bigcup_{n \in \N} D_n$ of compact invariant sets,
each admitting a partially hyperbolic splitting of the tangent bundle,
with zero center exponents.
Of course it would be nicer if we could conclude that $\mu(Z)=1$ or $D_n = M$ for some $n$.
That is the case if one of the following holds:
\begin{itemize}
\item if $f$ happens to be ergodic;
\item if $\dim M = 2$: then we must have $\mu(Z) = 1$ (so we reobtain the main result from~\cite{B GZLE});
\item if some $D_n$ has non-empty interior:
since the generic $f$ is transitive by \cite{Arnaud BC},
we conclude that $D_n=M$.
\end{itemize}

There is a fourth situation where we can improve the conclusions of Corollary~\ref{c.main}:
when considering globally partially hyperbolic diffeomorphisms,
that is, those that have a partially hyperbolic splitting defined on the whole tangent bundle.
(See \S\ref{ss.review} for the definition.)
There is no need to stress their relevance;
see e.g.\ the surveys \cite{Hass Pesin}, \cite{RRU survey}.

Let $\PH^1_\omega(M)$ indicate the (open) subset of $\Diff^1_\omega(M)$ formed by partially hyperbolic maps.
Then we have:
\begin{thm}\label{t.global PH}
For the generic $f$ in $\PH^1_\omega(M)$,
there is a partially hyperbolic splitting
$TM = E^u \oplus E^c \oplus E^s$
such that all Lyapunov exponents in the center bundle
vanish for a.~e.\ point.
\end{thm}

If the partially hyperbolic map $f$ belongs to the residual set given by Corollary~\ref{c.main},
then to get the conclusion of Theorem~\ref{t.global PH}
we have to ensure that $\dim E^0(x)$ is almost everywhere constant.
In the lack of ergodicity, the key property we use is \emph{accessibility},
which is known to be $C^1$ open and dense, by~\cite{Dolgo Wilk}.
See Section~\ref{s.PH proof} for the detailed proof.

\medskip

Let us now discuss briefly the topic of abundance of ergodicity,
and the relevance of Theorem~\ref{t.global PH} in this context.

An important problem in the literature is to determine geometric conditions
on a volume preserving dynamics that imply ergodicity of the Lebesgue measure.
Partial hyperbolicity seems to be a natural condition to start with.
Maybe not much more is needed:
Pugh and Shub conjectured in \cite{Pugh Shub JEMS}
that ergodic maps must form a $C^2$-open and dense set
among the partially hyperbolic ones.

\begin{rem}\label{r.dominance and ergodicity}
A more natural (but more difficult) condition to be imposed in the search for ergodicity is
the existence of a global dominated splitting.
That is so because this condition is satisfied for stably ergodic maps\footnote{A
(volume-preserving or symplectic) diffeomorphism
$f$ is called \emph{stably ergodic} if
it is of class $C^2$ and every $C^2$ (volume-preserving or symplectic)
map sufficiently $C^1$-close to $f$ is ergodic.}
(see \cite{Arbieto Matheus})
and there exist stably ergodic diffeomorphisms that are not partially hyperbolic (see \cite{Tahzibi}).
The situation for \emph{symplectic} maps is simpler,
because partial hyperbolicity is the same as dominance.
Stably ergodic symplectomorphisms are indeed partially hyperbolic, see \cite{Horita Tah, SX robust trans}.
\end{rem}

Improving significantly the results of Pugh and Shub~\cite{Pugh Shub JEMS},
Burns and Wilkinson \cite{Burns Wilk} gave the following list of conditions that are sufficient for ergodicity:
partial hyperbolicity, $C^2$ smoothness,  essential accessibility, and center bunching.
The latter condition
roughly means that the derivative restricted to the center bundle is close to conformal.

On the other hand, Theorem~\ref{t.global PH} says that
generic maps in $\PH_\omega^1(M)$ have a \emph{non-uniform center bunching} property
(which by semicontinuity is transmitted to nearby $C^2$ maps).
It is natural to ask if this property has interesting consequences.
Indeed it does:
non-uniform center bunching is used in \cite{ABW} to prove that
\emph{generic diffeomorphisms in $\PH_\omega^1(M)$ are ergodic.}

\medskip

Let us close this introduction
with a few comments on the choice of the topology.\footnote{Here I borrowed some arguments
from \cite{Avila smoothening}.}
For $C^r$ topologies with $r\ge 2$,
the perturbations we make in this paper definitely do not apply, and
indeed the main results do not extend.

The knowledge of $C^1$-generic dynamics has seen recently very significant progress; see
Chapter~10 of \cite{BDV livro} and the references therein.
Despite the fact that some fundamental questions are still open, a broad understanding is perhaps starting to emerge.
In contrast, few generic properties are known for topologies $C^r$ with $r>1$
(with the notable exception of one-dimensional dynamics):
even the Closing Lemma is open.

Sometimes $C^1$-generic and smoother behaviors are much different.
This is especially true for measure-theoretical properties related to distortion.
Despite these differences,
concrete examples and phenomena that arise from the study of $C^1$-dynamics
often turn out to be important in smoother contexts.
Some situations that illustrate this point are:
\begin{itemize}
\item
The concept of dominated splitting in dynamical systems originated from
the research of Liao 
and Ma\~{n}\'{e} 
on the Smale $C^1$-stability conjecture.
It is increasingly important in smooth ergodic theory:
see e.g.~\cite{ABV mostly exp}, also Remark~\ref{r.dominance and ergodicity}.

\item
The proof \cite{Dolgo Pesin} that
for every compact manifold other than the circle
there is a volume-preserving Bernoulli diffeomorphism
uses $C^1$-perturbation techniques from \cite{B GZLE}.

\item
The blenders introduced in~\cite{BD blenders}
to create new examples of $C^1$-robustly transitive diffeomorphisms
now appear as a ingredient for ergodicity in~\cite{RRTU 2d center}.
\end{itemize}

\section[Preliminaries and Plan of the Proof]{Preliminaries and Plan of the Proof}

\subsection[Review on dominated and partially hyperbolic splittings]{Review on Dominated and Partially Hyperbolic Splittings} \label{ss.review}

Let $f\colon M \to M$ be a $C^1$ diffeomorphism, and let $\Lambda \subset M$ be an $f$-invariant set.

A splitting $T_\Lambda M = E \oplus F$
is called \emph{$m$-dominated}, where $m\in\N$, if
it is $Df$-invariant, the dimensions of $E$ and $F$ are constant and positive,
and\footnote{The \emph{co-norm} of a linear map $A$ is $\m(A) = \inf_{\|v\|=1} \|Av\|$;
it equals $\|A^{-1}\|^{-1}$ if $A$ is invertible.}
$$
\frac{\|Df^m | E(x)\|}{\m(Df^m | F(x))} \le \frac{1}{2} \quad \text{for all $x\in \Lambda$.}
$$
We call $T_\Lambda M = E \oplus F$ a \emph{dominated splitting} if it is $m$-dominated for some $m$.
We also say that \emph{$E$ dominates $F$}.
The dimension of $E$ is called the \emph{index} of the splitting.

More generally, a $Df$-invariant splitting  $T_\Lambda M = E^1 \oplus \cdots \oplus E^k$
into non-zero bundles of constant dimensions
is called \emph{dominated} if $E^1 \oplus \cdots \oplus E^j$
dominates $E^{j+1} \oplus \cdots \oplus E^k$ for each $j<k$.
This definition coincides with the one \eqref{e.def DS adapted} given at the Introduction,
due to a result of Gourmelon~\cite{Gourmelon}.

A dominated splitting over the invariant set $\Lambda$ extends continuously to its closure;
so $\Lambda$ can be assumed to be compact when necessary.
See e.g.\ \cite{BDV livro} for the proof of this and other properties of dominated splittings.

\medskip

A $Df$-invariant splitting $T_\Lambda M = E^u \oplus E^c \oplus E^s$ is called \emph{partially hyperbolic}
if it is dominated, the bundle $E^u$ is uniformly expanding, and the bundle $E^s$ is uniformly contracting.
The latter two conditions mean that there is a uniform $m \in \N$ such that
$\m(Df^m| E^u) \ge 2$ and $\|Df^m | E^s \| \le \tfrac{1}{2}$ on $\Lambda$.
As it is customary, we extend the definition of partial hyperbolicity to allow $E^c$ to be $\{0\}$,
that is, to include uniform hyperbolicity.

Let's us mention an equivalent definition of partial hyperbolicity that is also frequent in the literature:
there is a Riemannian metric $\|\mathord{\cdot}\|$ on $M$
(called an \emph{adapted metric})
and continuous functions  $\alpha$, $\beta$, $\gamma$, $\delta$ on
the compact set $\Lambda$ such that
the following inequalities hold at each point of $\Lambda$:
\begin{equation}\label{e.adapted again}
\begin{gathered}
\alpha > 1 > \delta \, , \\
\m (Df | E^u) \ge \alpha > \beta > \| Df|E^c \|
\ge \m (Df|E^c) \ge \gamma > \delta \ge \|Df|E^s\| \, .
\end{gathered}
\end{equation}
The equivalence of the two definitions is shown in \cite{Gourmelon}.

\begin{rem}\label{r.definitions PH}
If one asks $\alpha$, $\beta$, $\gamma$, $\delta$ in \eqref{e.adapted again} to be constants,
then one has a stronger notion of partial hyperbolicity, called \emph{absolute}.
The weaker notion used in this paper is called \emph{relative} (or \emph{pointwise}) partial hyperbolicity.
See~\cite{AV flavors} for a detailed discussion.
\end{rem}

The precise meaning of the sentence
``dominated splittings are automatically partially hyperbolic in the symplectic case''
is:

\begin{othertheorem}[Theorem~11 in \cite{BV Cambridge}] \label{t.DS is PH}
Let $f$ be a symplectic diffeomorphism
and let $T_\Lambda M = E \oplus F$ be a dominated splitting over a $f$-invariant set $\Lambda$.
Assume $\dim E \le \dim F$ and let $E^u =E$.
Then $F$ splits invariantly as $E^c \oplus E^s$ with $\dim E^u = \dim E^s$,
and the splitting $T_\Lambda M = E^u \oplus E^c \oplus E^s$ is partially hyperbolic.
\end{othertheorem}

\begin{othertheorem}[Corollary~B.1 in \cite{BV Cambridge}] \label{t.0 or 1}
A hyperbolic set of a generic symplectic diffeomorphism has either zero or full volume.
\end{othertheorem}

It is now easy how Corollary~\ref{c.main} reduces to Theorem~\ref{t.main}:

\begin{proof}[Proof of Corollary~\ref{c.main}]
By Theorem~\ref{t.0 or 1}, there is a residual subset $\cR_1 \subset \Diff^1_\omega(M)$
formed by maps that either are Anosov
or have no hyperbolic sets of positive measure.
Let $\cR_2$ be residual set given by Theorem~\ref{t.main},
and let $f \in \cR_1 \cap \cR_2$.
By Theorem~\ref{t.DS is PH},
the zipped Oseledets splitting along the orbit of a.e.\ point $x$
is either uniformly hyperbolic (if $\dim E^0(x) = 0$),
or partially hyperbolic with $3$ non-zero bundles (if $2 \le \dim E^0(x) \le 2N-2$),
or trivial (if $\dim E^0(x) = 2N$).
The first option occurs for a positive measure set if and only if $f$ is Anosov.
So $f$ satisfies the stated conclusions.
\end{proof}

\subsection[Discontinuity of the Lyapunov exponents]{Discontinuity of the Lyapunov Exponents}
\label{ss.discontinuity}

Given $f\in \Diff_\omega^1(M)$ and a regular point $x \in M$,
rewrite the list of Lyapunov exponents in non-increasing order and repeating each according to its multiplicity:
$$
\lambda_1(f,x) \ge \cdots \ge \lambda_{2N}(f,x)
$$

For $p = 1, \ldots, N$, we consider the \emph{integrated $p$-exponent} of the
diffeomorphism~$f$:
$$
\LE_p (f) = \int_M \big(\lambda_1(f,x) +\cdots + \lambda_p(f,x) \big) \, d\mu(x).
$$
The map $\LE_p\colon  \Diff_\omega^1(M) \to \R$ is upper-semicontinuous,
and therefore its points of continuity constitute a residual subset of $\Diff_\omega^1(M)$.
On the other hand, continuity of the integrated exponents has strong consequences:

\begin{thm}\label{t.continuity}
Let $f \in \Diff_\omega^1(M)$ be such that each map $\LE_1, \ldots, \LE_N$
is continuous at~$f$.
Then for $\mu$-almost every $x\in M$,
the Oseledets splitting of $f$ is either dominated or
trivial along the orbit of $x$.
\end{thm}
The main result we prove is Theorem~\ref{t.continuity},
and Theorem~\ref{t.main} is itself an immediate corollary.
Theorem~\ref{t.continuity} has a more quantitative version,
Proposition~\ref{p.jump}, which is used in the proof of Theorem~\ref{t.global PH}.

\subsection[A preview of the proof]{A Preview of the Proof} \label{ss.preview}

This subsection contains
an informal outline of the proof of Theorem~\ref{t.continuity}.
It is logically independent from the rest of the paper.
However, it should help the reader to go through the complete proof.

\medskip

Assume that the Oseledets splitting of a symplectic diffeomorphism $f$
is non-trivial and not dominated.
To prove Theorem~\ref{t.continuity} (and hence \ref{t.main}),
we need to show that for some $p$, the integrated exponent $\LE_p$ is discontinuous at $f$.
The proof has two parts:
\begin{enumerate}
\item Assume that the Oseledets splitting $T_{\mathrm{orb(x)}} M = E^1 \oplus \cdots \oplus E^k$
along the orbit of some point $x$ is non-trivial and not dominated:
that is, for some $i$, $E = E^1 \oplus \cdots \oplus E^i$ does not dominate
$F = E^{i+1} \oplus \cdots \oplus E^k$.
Let $p = \dim E$; for symplectic reasons
it suffices to consider the case $p \le N =\tfrac{1}{2}\dim M$.

Some positive iterate $y$ of $x$ will enter a zone where
the non-dominance of the splitting $E\oplus F$
manifests itself. (More on this later.)
Then one constructs by hand a $C^1$-perturbation $g$ of $f$
with the following properties:
For some $m \in \N$,
$Dg^m(y)$ sends some (non-zero) vector in the space $E$ into the space $F$.
The support of the perturbation is a small neighborhood
${U \sqcup f(U) \sqcup \cdots \sqcup f^{m-1}(U)}$ (called a tower)
of the orbit segment $\{y, \ldots, f^{m-1} y\}$.
Furthermore, it is important that some vectors from
$E(\tilde y)$ are sent by $Dg^m(\tilde y)$ into $F(\tilde y)$ not only
at the point $\tilde y = y$, but also for most (in the sense of measure) points $\tilde y$ in
the base $U$ of the tower.

\item The global procedure is to cover most of the manifold by many disjoint tall and thin towers.
Approximately in the middle of each tower, a perturbation as sketched in part 1 above is performed.
The result is the different expansion rates of $E$ and $F$ are blended,
and the integrated $p$-exponent of the new diffeomorphism dropped.
So one concludes that $\LE_p$ is discontinuous at $f$, as desired.
\end{enumerate}

This general strategy is the same followed in the papers \cite{B GZLE} and \cite{BV Annals}.
More detailed (and still informal) descriptions of it
can be found in \cite{BV IHP} and \cite{BV Cambridge}.
It is clear that the methods would fail for topologies finer than~$C^1$.

\medskip

To explain the difficulties of the symplectic case, let us return to the first step of the strategy,
and look more closely
how the non-dominance of the splitting $E \oplus F$ manifests itself at the point $y$.
There are four possibilities:
\begin{enumerate}
\item[I.]
Either the angle $\angle (E,F)$ gets very small at $y$.

\item[II.]
Or there is some $m \in \N$ and there are
unit vectors $v \in E(y)$, $w \in F(y)$ such that $w$ gets much more expanded than $v$
by $Df^m(y)$.

\item[III.]
Or there is some large $m \in \N$ and
there are non-zero vectors $v \in E(y)$ and $w \in F(y)$ with $\omega(v,w) \neq 0$ and
such that no vector in the plane $P$ spanned by them
gets much expanded nor contracted by $Df^j(y)$ for all $j=1$, \ldots, $m$.
This means that after a bounded change of the Riemannian metric,
the restriction of $Df^j(y)$ to $P$ becomes an isometry,
for all $j=1,\ldots, m$.
Notice the symplectic form $\omega$ restricted to $P$ is non-degenerate (because $\omega (v,w)\neq 0$).

\item[IV.]
Or there is some large $m \in \N$ and
there are non-zero vectors $v \in E(y)$ and $w \in F(y)$
spanning a plane $P$ that is (up to time $m$)
\emph{uniformly expanding} and \emph{conformal}.
That is, there exists $\tau>1$ such that
after a bounded change of the Riemannian metric we have
that $Df^j (y)/ \|Df^j(y)\|$ is an isometry and
$\|Df^j(y) \| \ge \tau^j$
for all $j=1,\ldots, m$.
Since the plane $P$ is expanded it must be \emph{null} (meaning that
the symplectic form vanishes on $P\times P$).
\end{enumerate}

Let us explain how in each case one sends a vector from $E$ into $F$ by perturbing $f$.
Since we will work on very small neighborhoods of a segment of orbit,
we can assume $f$ is locally linear.

In case~I, one composes $f$ with a small rotation supported around $y$.
Let us be a little more precise.
If $\dim M =2$, pretend $M= \R^2$ and $y=0$, and let $\alpha = \angle(E(y), F(y))$;
then the perturbation will be given by $g(x) = f(R_{\theta(x)} (x))$,
where $\theta$ vanishes outside a small disk $D = B_r(0)$ and
is constant equal to $\alpha$ on a smaller $D_1 = B_{r_1}(0)$.
It is very important
that the measure of the \emph{buffer} $D \setminus D_1$ is small compared that
of the support $D$.
For $\dim M > 2$, the rotation is made around a codimension $2$ axis,
and disks are replaced by cylinders.

The second case is similar: we make two rotations, one around $y$ and other around $f^m y$.

Case~III is more delicate: one has to make small rotations
around each of the points $y$, $fy$, \ldots, $f^{m-1} y$.
The rotations must be \emph{nested}, that is,
the buffer of each rotation is mapped by $f$ to the next buffer.
(This is necessary to control the measure of the set where the perturbation will be effective.)
Since the ambient space $M$ has dimension $2N>2$,
each rotation is around an $(2N-2)$-dimensional axis $X$,
and the actual support is a thin cylinder along $X$.
Moreover, in order to preserve the symplectic form,
$X$ needs to be the symplectic complement of the plane $P$.
Thus the fact that $\omega$ is non-degenerate on $P$ is also used.

The treatment of the first three cases explained above is the same as in \cite{BV Annals}.
In fact, case IV does not occur if $\dim E = \dim F$.
That is the precise reason why it does not appear in \cite{BV Annals}.
(Let us remark that in the \emph{volume-preserving} situation dealt with in \cite{BV Annals}
there are only three cases, similar to those explained above.
The construction of the nested rotations has some extra subtleties, however.)

The main novelty of the present paper is a
perturbation method that permits us to treat the case~IV.
Before explaining it, let us see what the difficulties are.

It seems natural to try nested rotations again in case~IV,
because $Df$ acts conformally on the plane $P$.
However, a linear map that rotates $P$ and is the identity on a space complementary to $P$
\emph{cannot} preserve the symplectic form.
The reason is that $P$ is a null space.
To preserve the symplectic form,
one also needs to rotate another $2$-dimensional space $Q$;
then the linear map can be taken as the identity on a certain ``axis'' of dimension $(2N-4)$
(that is the symplectic complement of $P \oplus Q$).
Thus the situation becomes essentially four-dimensional.
Indeed, let us from now on assume $\dim M = 4$ (and pretend that $M=\R^4$)
to simplify the discussion.
Therefore $\dim E = 1$ and $\dim F = 3$.

Standard symplectic coordinates $p_1$, $p_2$, $q_1$, $q_2$ on $\R^4$
can be found with the following properties:
the $p_1 p_2$ and $q_1 q_2$-planes are $P$ and $Q$, respectively,
$E$ is the $p_1$ axis, and $F$ is the space $p_2 q_1 q_2$.
Moreover, the derivatives take the following form:
$$
Df(f^i y)\colon (p_1, p_2, q_1, q_2) \mapsto
\big(\tau_i p_1, \tau_i p_2, \tau_i^{-1} q_1, \tau_i^{-1} q_2 \big) , \quad
\text{where } \tau_i \ge \tau > 1
$$
(for $1 \le 1 \le m$.)
So the splitting $P \oplus Q$ has a uniformly hyperbolic behavior:
$P$ is expanded and $Q$ is contracted.

Now start with a nice domain $D$ (say, a disk in the plane $P$
times a disk of the same size in the plane $Q$)
for the support for the first perturbation.
By the uniform hyperbolicity of the splitting $P\oplus Q$,
the images $Df^i(y) (D)$ get quickly very deformed.
Nesting means that the effective support (that is, the support minus the buffer)
of each perturbation is the $f$-image of the previous one.
But the perturbations must also be $C^1$-small, so it becomes hard to rotate $P$ and $Q$
by a fixed angle.
This is the main obstacle for the use of nested rotations in case~IV.
(And there is another, more subtle, obstacle: if the support is a box $D$ as above,
it is unclear how to rotate by a constant angle while keeping a small buffer.
That is because the rotations we want arise
from the linear flow generated by the hamiltonian $H = p_2 q_1 - p_1 q_2$,
and since this quadratic form has no definite sign, it cannot
be flattened outside of $D$ like in the proof of Lemma~5.5 from \cite{BV Annals}.)

Finally, let us explain the main idea.
We abandon nested rotations and buffers.

Start with a small box neighborhood $D$ of $y$ as above,
and consider the field of directions $v_0$ spanned by the
constant vector field $\vetor{p_1}$.
Due to the hyperbolicity of the splitting $P \oplus Q$,
there is a strictly invariant cone around the expanding space $P$.
(Of course the cone field will be also invariant
under a perturbation $g$ of $f$.)
Given two directions in the cone, we project them on $P$ along $Q$,
and measure the obtained oriented angle;
let us call this the $p_1 p_2$-angle between the two directions.
Notice $f$ preserves $p_1 p_2$-angles.

Take a symplectic diffeomorphism $h_0\colon \R^4 \to \R^4$
that is $C^1$-close to the identity, is the identity outside of $D$,
and does \emph{not} leave the field $v_0$ invariant.
The perturbation of $f$ in the neighborhood of $y$ is $g = f \circ h_0$.
\emph{Any} $h_0$ with those properties works, and will be the base for the rest of the construction.

The perturbation around $f(y)$ must be supported on $f(D) = g(D)$.
On $g(D)$ we have a field of directions $v_1$ that is the image of
the constant field $v_0$ by $Dg$.

Then take many disjoint boxes $D_i \subset g(D)$ covering
all of $g(D)$, except for a set of very small measure.
The boxes are taken so small so that the variation of the field $v_1$
on each of them is very small.
So let us pretend that the linefield $v_1$ is constant in each $D_i$.
(See Figure~\ref{f.bla}.)

\begin{figure}[htb]
\begin{center}
\psfrag{p1p2}{{$p_1 p_2$}}
\psfrag{q1q2}{{$q_1 q_2$}}
\psfrag{hor}{{$\vetor{p_1}$}}
\psfrag{g}{{$g$}}
\includegraphics[width= {\textwidth}]{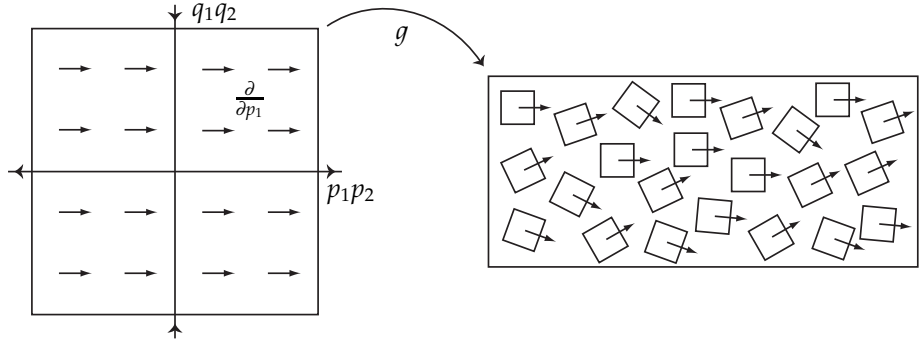}
\caption{First step of the perturbation: disjoint boxes $D_i$ cover
most of the image of the box $D$.}
\label{f.bla}
\end{center}
\end{figure}

Each $D_i$ is a shrunk copy of $D$: there is an affine map $T_i \colon D_i \to D$
that takes $v_1$ to $v_0$.
Let $h_1$ be a map that equals $T_i^{-1} \circ h_0 \circ T_i$ on each $D_i$,
and the identity outside of $\bigcup D_i$.
With the necessary precautions, $h_1$ becomes symplectic and $C^1$-close to the identity.
Now define the perturbation $g$ on $g(D)$ as equal to $f\circ h_1$.

Let $X_0$ and $X_1$ be the $p_1p_2$-angles turned
in the first and second steps, respectively.
That is, for $x\in D$, let $X_0 (x)$ be the (oriented)
$p_1p_2$-angle between $v_0$ and $Dh(x) \cdot v_0$,
and let $X_1(x)$ be the (oriented) $p_1 p_2$-angle between
$v_1 (g(x))$ and ${Dh_1( g(x)) \cdot v_1 (g(x))}$.
Notice that $X_0$ is not identically zero by construction.
Since the linefield $v_0$ is $Df$-invariant,
the $p_1 p_2$-angle between $v_0$ and ${Dg(x) \cdot v_0}$ equals $X_0$.
Also, the $p_1 p_2$-angle between $v_0$ and ${Dg^2(x) \cdot v_0}$
is $X_0 + X_1$.

Let us re-scale Lebesgue measure $\mu$ so that $\mu(D)=1$.
So $X_0$ and $X_1$ can be thought as a \emph{random variables}.
The key observation is that
they are \emph{independent and identically distributed}.

We continue in an analogous way: in the next step we cover each $g(D_i)$
by still smaller boxes $D_{ij}$, each of them so that the
field of directions $v_2 = Dg \cdot v_1$ is almost constant.
In each $D_{ij}$ the perturbation $g$ is modeled on the map $h_0$ as described above.
Continuing in this way, we obtain sequences of maps $g \colon  g^i(D) \to g^{i+1}(D)$ and
i.i.d.\ random variables $X_i$ such that
$Dg^n$ turns the vector $\vetor{p_1}$ by an angle $S_n = X_0 + \cdots + X_{n-1}$ in the
$p_1 p_2$-plane.

This construction gives a \emph{random walk} $S_n$ on the real line.
The probability that a path of the random walk says for all time confined in some compact interval is zero.
Moreover, the steps $X_n$ are small.
Thus for almost every orbit
there is a first time the angle $S_n$ becomes close to $\pm \pi/2$.
Then we modify the construction: we perturb one last time to make the angle exactly $\pm \pi/2$,
and then perturb no more along that orbit.
In other words, the angles behave as a random walk with absorbing barriers around $\pm \pi/2$.

The conclusion is that in some large but finite time,
for the majority of orbits of $g$, the images of the vector $\vetor{p_1}$ in $E$
eventually have $p_1 p_2$-angle equal to $\pm \pi/2$,
and this means the $1$-dimensional space $E$ has been sent into the $3$-dimensional space $p_2q_1q_2$, that is, $F$.
So the perturbation $g$ has the desired properties, and case~IV is settled.

\subsection[Organization of the rest of the paper]{Organization of the Rest of the Paper}

As explained in \S\ref{ss.preview},
the proof of Theorem~\ref{t.continuity} splits into a local and a global part.
The local part of the proof takes Sections~\ref{s.flex 0} to \ref{s.flex II}.

In Section~\ref{s.flex 0} we introduce the ad hoc concept of flexibility,
which summarizes the properties our perturbations need to have.
(Namely, to make two bundles of a splitting collide for a set of points of large measure.)
Flexibility replaces the notion of realizable sequences from \cite{BV Annals},
which is not sufficient for our purposes.

In Section~\ref{s.flex I} we show that lack of dominance
can be classified in four types.
The proof consists of symplectic linear algebra.

In Section~\ref{s.flex II} we show that each of the four cases has the desired flexibility property.
The fourth case is dealt with in \S\ref{ss.IV},
where the probabilistic method for the construction of the perturbations is
explained in detail.

In Section~\ref{s.flex III} we complete the proof of Theorem~\ref{t.continuity}
giving its global part.
This part is essentially contained in \cite{BV Annals},
but we will present a simplified proof.

In the final Section~\ref{s.PH proof} we prove Theorem~\ref{t.global PH}.

\section{Flexibility}\label{s.flex 0}

\subsection[Split sequences on $\R^{2N}$ and the flexibility property]{Split Sequences on $\R^{2N}$ and the Flexibility Property}

Let $N$ be fixed.
We consider $\R^{2N} = \{(p_1,\ldots,p_N, q_1,\ldots, q_N)\}$
endowed with the standard symplectic form $\omega = \sum_i dp_i \wedge dq_i$,
and with Lebesgue measure~$\mu$.
The euclidian norm on $\R^{2N}$ and also the induced operator norm are indicated by $\|\mathord{\cdot}\|$.

\smallskip

A \emph{split sequence}
of length $n$ is composed of the following objects:
\begin{itemize}
\item a (finite) sequence of linear $\omega$-preserving maps
$$
\R^{2N} \xrightarrow{A_0} \R^{2N} \xrightarrow{A_1} \cdots \xrightarrow{A_{n-1}} \R^{2N}
$$

\item non-trivial linear  splittings $\R^{2N} = E^1_i \oplus E^2_i$, for $0 \le i \le n$,
that are invariant in the sense that
$A_i \cdot E^*_i = E^*_{i+1}$, $* =1$, $2$.
\end{itemize}
The constant $p = \dim E_i$ is called the \emph{index} of the split sequence.

Let $\eps>0$ and $\kappa>0$.
We say that a split sequence
$\{A_i, E_i^{1,2}\}$ of length $n$
is \emph{$(\eps, \kappa)$-flexible} if for every $\gamma>0$,
there exists a bounded open neighborhood $U$ of $0$ in $\R^{2N}$ and
there exist symplectomorphisms $g_0$, \ldots, $g_{n-1}\colon  \R^{2N} \to \R^{2N}$
such that:
\begin{enumerate}
\item $g_i$ equals $A_i$ outside $A_{i-1} \ccircs A_0(U)$ for each $i=0,\ldots,n-1$;\footnote{$\ccircs = \circ \cdots \circ$}
\item $\| D(A_i^{-1} \circ g_i) - \Id\| < \eps$ uniformly, for each $i=0,\ldots,n-1$;
\item there is a set $G \subset U$ such that $\mu(G) > (1-\kappa) \mu(U)$ and
\footnote{The angle $\angle(E,F) \in [0, \pi/2]$ between non-zero linear subspaces
$E$, $F \subset \R^{2N}$ is defined as the minimum of the angles $\angle(v,w)$ over non-zero vectors $v\in E$, $w\in F$.}
$$
\angle \left( D(g_{n-1} \ccircs g_0) (x) \cdot E^1_0, \, E^2_n \right) < \gamma
\quad \forall x \in G \, .
$$
\end{enumerate}

Informally, the linear maps $A_i$ can be (non-linearly) perturbed so that
the space $E^1$ is sent after time $n$ very close to the space $E^2$,
for most points in the support of the perturbation.

\begin{rem}
Flexibility appears implicitly in~\cite{BV Annals}. 
The main difference is that in all situations considered there,
the map $x \mapsto D(g_{n-1} \ccircs g_0)(x)$ is
approximately (to error $\gamma$) constant on $G$.
This will not be always the case here.
\end{rem}

Loosely speaking, the next lemma says that flexibility is preserved by
changes of coordinates.

\begin{lemma}\label{l.flex change coord}
Consider two split sequences of the same length:
$$
\left\{ E_i^1 \oplus E_i^2 \xrightarrow{A_i} E_{i+1}^1 \oplus E_{i+1}^2 \right\}_{0\le i <n}
\quad \text{and} \quad
\left\{ F_i^1 \oplus F_i^2 \xrightarrow{B_i} F_{i+1}^1 \oplus F_{i+1}^2 \right\}_{0\le i <n}
$$
Assume that there are linear symplectic maps
$C_0$, \ldots, $C_n\colon  \R^{2N} \to \R^{2N}$ such that
${C_{i+1} \circ A_i} = {B_i \circ C_i}$ and $C_i (E_i^*) = F_i^*$.
Let $K = \max_i \|C_i\|$.
If the split sequence $\{A_i, E^{1,2}_i\}$ is $(\eps, \kappa)$-flexible then
$\{B_i, F^{1,2}_i\}$ is $(K^2\eps, \kappa)$-flexible.
\end{lemma}

\begin{proof}
The proof is straightforward, but let us give it anyway.
Given $\gamma>0$, let $U$, $g_i$, and $G$ be given by the $(\eps, \kappa)$-flexibility of
the sequence $\{A_i, E^{1,2}_i\}$.
Define $\hat U = C_0(U)$, $\hat{g}_i = C_{i+1} \circ g_i \circ C_i^{-1}$, and $\hat{G} = C_0(G)$.
Let us check that these objects satisfy conditions 1, 2, and 3
in the definition of $(K^2\eps, \kappa)$-flexibility.
The first one is obvious.
Since the linear map $C_i$ is symplectic, $\|C_i\| = \|C_i^{-1}\|$ and so
$$
\big\| D(B_i^{-1} \circ \hat{g}_i) - \Id \big\| \le
\big\| C_i \circ \big(D(A_i^{-1} \circ g_i) - \Id\big) \circ C_i^{-1} \big\| < K^2\eps,
$$
which is condition 2.
Given $y \in \hat{G}$, let $x = C_0^{-1}(y)$.
The spaces
$D(\hat{g}_{n-1} \ccircs \hat{g}_0) (y) \cdot F^1_0$ and $F^2_n$
are the respective images by $C_n$ of the spaces
$D(g_{n-1} \ccircs g_0) (x) \cdot E^1_0$ and $E^2_n$.
The angle between the latter pair of spaces is less than $\gamma$,
therefore the angle formed by the earlier pair is at most $K' \gamma$,
where $K'= K'(K)$.
(In fact, $K'= \tfrac{\pi}{2} K^2$ works: see \cite[Lemma~2.7]{BV Annals}.)
Since $\gamma>0$ was arbitrarily chosen, condition 3 is verified.
\end{proof}

The following lemma is trivial:

\begin{lemma}\label{l.short seq}
Let $\left\{ E_i^1 \oplus E_i^2 \xrightarrow{A_i} E_{i+1}^1 \oplus E_{i+1}^2 \right\}_{0\le i <n}$
be a split sequence.
If there are $0\le i_0 < i_1 \le n$ such that the shorter split sequence
$\left\{ E_i^1 \oplus E_i^2 \xrightarrow{A_i} E_{i+1}^1 \oplus E_{i+1}^2 \right\}_{i_0\le i < i_1}$
is $(\eps,\kappa)$-flexible, then so is the whole split sequence of length $n$.
\end{lemma}

The next lemma says that the domain $U$ in the definition of flexibility can be chosen arbitrarily.

\begin{lemma} \label{l.choose U}
Assume that 
$\{A_i, E_i^{1,2}\}$
is a $(\eps, \kappa/2)$-flexible split sequence of length $n$.
Then for \emph{any} non-empty bounded open set $U \subset \R^{2N}$ 
and any $\gamma>0$, there exist maps $g_0$, \ldots, $g_{n-1}\colon  \R^{2N} \to \R^{2N}$
satisfying the three conditions in the definition of $(\eps,\kappa)$-flexibility.
\end{lemma}

\begin{proof}
Given $\gamma>0$,
the $(\eps, \kappa/2)$-flexibility of the splitting sequence $\{A_i, E_i^*\}$
provides a set $\hat{U}$ and symplectomorphisms $\hat{g}_0$, \ldots, $\hat{g}_{n-1}$
with the following properties:
(1) each $\hat{g}_i$ equals $A_i$ outside $A_{i-1}\ccircs A_0(U)$;
(2) the derivative of $A_i^{-1} \circ \hat{g}_i$ is $\eps$-close to the identity;
and (3) the image of $E^1_0$ by the derivative of $\hat{g}_{n-1} \ccircs \hat{g}_0$ is $\gamma$-close to $E_n^2$
for all points in a set $\hat{G}$ with measure at least $(1-\kappa/2)\mu(\hat{U})$.

Now fix some non-empty bounded open set $U$.
By the Vitali Covering Lemma, we can find a finite family of disjoint sets $\hat{U}_j \subset U$
such that the measure of $U \setminus \bigsqcup_j \hat{U}_j$ is less than $\frac{\kappa}{2}\mu(U)$,
and each $\hat{U}_j$ is equal to $T_j(\hat{U})$,
where $T_j \colon  \R^{2N} \to \R^{2N}$ is a homothety composed with a translation.

For $i=0,\ldots,n-1$, let
$$
T_{j,i} = A_{i-1} \ccircs A_0 \circ T_j \circ (A_{i-1} \ccircs A_0)^{-1} .
$$
Of course, $T_{j,i}$ is a homothety composed with a translation.
Define $g_i \colon  \R^{2N} \to \R^{2N}$
as equal to $A_i$ outside $\bigsqcup_j A_{i-1} \ccircs A_0 (\hat{U}_j)$
and equal to
$$
A_i \circ T_{j,i} \circ A_i^{-1} \circ \hat{g}_i \circ T_{j,i}^{-1}
$$
inside each $A_{i-1} \ccircs A_0(\hat{U}_j)$.
Let us see that these maps satisfy the three conditions in the definition of $(\eps,\kappa)$-flexibility.
The first one is obvious.
We have $D(A_i^{-1} \circ g_i) (x) = D(A_i^{-1} \circ \hat{g}_i) (T_{j,i}^{-1}(x))$,
on $A_{i-1} \ccircs A_0(\hat{U}_j)$,
so the second condition holds (and $g_i$ is symplectic).
Finally, let $G = \bigsqcup_j T_j(\hat{G})$
Then $\mu(G) > (1-\kappa/2)^2 \mu(U)$.
Moreover,
the image of $E^1_0$ by the derivative of
$$
g_{n-1} \ccircs g_0 = T_{j,n} \circ \hat{g}_n \ccircs \hat{g}_0 \circ T_{j,0}^{-1}
$$
is $\gamma$-close to $E_n^2$ for all points in $T_j(\hat G) \subset G$.
This proves condition~3.
\end{proof}

\subsection[Flexibility on the tangent bundle]{Flexibility on the Tangent Bundle}\label{ss.flex TM}

Let $M$ be a fixed closed symplectic manifold of dimension $2N$.
By Darboux' Theorem, there exists an atlas $\{\phi_i\colon  V_i \to \R^{2N}\}$
formed by charts that take the symplectic form on $M$
to the standard symplectic form on $\R^{2N}$.
Let $K_\cA > 1$ be such that
such an atlas can be chosen with $\|D\phi_i\|$, $\|D\phi_i^{-1}\| < K_\cA$ everywhere.
\emph{Fix $K_\cA$ once and for all, and let $\cA$ be the maximal symplectic atlas
obeying the bounds above.}
That is, $\cA$ is the set of all symplectic maps
$\phi \colon  V \to \R^{2N}$, where $V \subset M$ is open,
such that $\|D\phi (x) \|< K_\cA$ for all $x\in V$ and
$\|D\phi^{-1}(y)\| < K_\cA$ for all $y \in \phi(V)$.

Choose a finite atlas $\cA_0 \subset \cA$.
For each $z \in M$, choose and fix some chart $\phi_z\colon  V_z \to \R^{2N}$ in $\cA_0$
with $V_z \ni z$.
For any $x\in V_z$, we define a linear isomorphism
\begin{equation}\label{e.i}
\mathrm{i}_x^z \colon  T_z M \to T_x M \quad \text{by} \quad
\mathrm{i}_x^z = [D\phi_z (x)]^{-1} \circ D\phi_z(z) \, .
\end{equation}

\smallskip

Now we extend the notions of split sequences and
flexibility to the tangent bundle $TM$.

Fixing $f\in \Diff_\omega^1(M)$ and a non-periodic point $z \in M$,
a \emph{split sequence on $TM$} is composed of the objects:
\begin{itemize}
\item the (finite) sequence of linear maps $Df(f^i z)$, where $0 \le i < n$;
\item non-trivial splittings $T_{f^i z} M = E_i^1 \oplus E_i^2$, for $0 \le i \le n$,
invariant in the sense that $Df (f^i z) \cdot E_i^* = E_{i+1}^*$.
\end{itemize}

Using charts, a split sequence on $TM$ induces a split sequence on $\R^{2N}$.
More precisely, for each $i =0,\ldots,n$, let
$\phi_i$ be a chart in the atlas $\cA$ whose domain contains $f^i z$.
Then we consider the split sequence on $\R^{2N}$
$$\left\{ \hat{E}_i^1 \oplus \hat{E}_i^2 \xrightarrow{A_i}
\hat{E}_{i+1}^1 \oplus \hat{E}_{i+1}^2 \right\}_{0\le i <n}
\text{where }
A_i = D (\phi_{i+1} \circ f \circ \phi_i^{-1})(\phi_i(f^i z)) , \
\hat{E}^*_i = D\phi_i(f^i z) \cdot E^*_i \, .
$$

A split sequence on $TM$ is called \emph{$(\eps,\kappa)$-flexible}
if so is a induced split sequence on $\R^{2N}$, for \emph{some} choice of the charts.

Given a split sequence on $TM$, we can find special perturbations of
the diffeomorphism $f$, as described in the lemma below:

\begin{lemma}\label{l.realization}
Given $f \in \Diff^1_\omega(M)$ and a neighborhood $\cV$ of $f$ in $\Diff^1_\omega(M)$,
there exists $\eps>0$ such that the following holds:
Let $z\in M$ be a non-periodic point for $f$.
Assume that
$
Df(f^i z)\colon E_i^1 \oplus E_i^2 \to E_{i+1}^1 \oplus E_{i+1}^2 \ (\ 0\le i <n)
$
is a $(\eps, \kappa)$-flexible split sequence.

Then for every $\gamma>0$ there exists $r>0$ with
the following properties:
First, the closed ball $\bar{B}_r(z)$ is disjoint from its $n$ first iterates.
Second, given any non-empty open set $U \subset B_r(z)$,
there exists $g \in \cV$
with the following properties:
\begin{enumerate}
\item $g$ equals $f$ outside $\bigsqcup_{i=0}^{n-1} f^i(U)$;
\item there is a set $G \subset U$ with $\mu(G) > (1-\kappa)\mu(U)$
such that
$$
\text{for every $x\in G$,} \quad
\angle \left( Dg^n(x) \; \mathrm{i}_x^z \cdot E^1_0, \,
              \mathrm{i}_{g^n x}^{f^n z} \cdot E^2_n \right)< \gamma \, .
$$
\end{enumerate}
\end{lemma}

\begin{proof}
Let $\eps = \eps(f, \cV)$ be small (to be specified later).

Let $z \in M$, $n\in \N$, $\kappa>0$,
and $T_{f^i z} M = E^1_i \oplus E^2_i$ be as in the assumptions of the lemma.
That is, there exist
charts $\phi_i\colon  V_i \to \R^{2N}$ (for $0 \le i \le n$) in the atlas $\cA$ such that
$V_i \ni f^i z$ and the split sequence $\{A_i, \hat{E}_i^*\}$ defined by
$$
A_i = D (\phi_{i+1} \circ f \circ \phi_i^{-1})(\phi_i(f^i z)), \qquad
\hat{E}^*_i = D\phi_i(f^i z) \cdot E^*_i
$$
is $(\eps, \kappa)$-flexible.
Without loss of generality, assume that $\phi_i (f^i z) = 0$ and
that $V_i = f^i (V_0)$.

We can also assume that \emph{the expression of $f$ in the charts is linear},
that is, $\phi_{i+1} \circ f \circ \phi_i^{-1}$ is the restriction
of the linear map $A_i$ to $\phi_i(V_i)$.
To see this, let $\psi_i = A_{i-1} \ccircs A_0 \circ \phi_0 \circ f^{-i}$,
for $0 \le i \le n$.
Then $\psi_i$ is a symplectomorphism from a neighborhood of $f^i z$
to a neighborhood of $0$ in $\R^{2N}$.
Also, it follows from the definition of the $A_i$'s that
$D\psi_i(f^i z) = D\phi_i(f^i z)$.
Therefore $\psi_i\colon  W_i \to \R^{2N}$ are charts
in the atlas $\cA$,
provided we choose sufficiently small neighborhoods $W_i$ of $f^i z$.
Moreover, $\psi_{i+1} \circ f \circ \psi_i^{-1}$ equals $A_i$ (where the former is defined).
So we just need to replace $\phi_i$ with $\psi_i$.

Now the proof becomes straightforward.
Let $\gamma>0$ be given.
Choose $r$ with $0< r< \eps$ such that
the closed ball $\bar{B}_r(z)$ is contained in $V_0$
and is disjoint from its first $n$ iterates.

Given a non-empty open set $U \subset B_r(z)$,
let $\hat{U} = \phi_0 (U)$.
Take $\gamma' \ll \gamma$.  
The flexibility of the split sequence $\{A_i, \hat{E}_i^{1,2}\}$,
together with Lemma~\ref{l.choose U},
implies that
there exist symplectomorphisms $g_i\colon  \R^{2N} \to \R^{2N}$ (for $0 \le i < n$)
such that:
\begin{enumerate}
\item $g_i$ equals $A_i$ outside $A_{i-1} \ccircs A_0(\hat{U}) = \phi_i(f^i(U))$;
\item $\|D (A_i^{-1} \circ g_i) - \Id \| < \eps$;
\item there is a set $\hat{G} \subset \hat{U}$ such that $\mu(\hat{G}) > (1-\kappa) \mu(\hat{U})$ and
$$
\angle \left( D(g_{n-1} \ccircs g_0) (\hat{x}) \cdot \hat{E}^1_0, \hat{E}^2_n \right) < \gamma'
\quad \forall \hat{x} \in \hat{G} \, .
$$
\end{enumerate}

Define $g\colon  M \to M$ by
$$
g(x) =
\begin{cases}
\phi_{i+1}^{-1} \circ g_i \circ \phi_i (x) &\text{if $x \in V_i = f^i(V_0)$, $0 \le i < n$,} \\
f(x) &\text{otherwise.}
\end{cases}
$$
Then $g$ is a symplectomorphism that equals $f$ outside $\bigsqcup_{i=0}^{n-1} f^i(U)$;
moreover if $\eps$ is small enough then $g$ is close to $f$, that is, $g \in \cV$.
Now, if $r$ is sufficiently small then
for every $x \in G = \phi_0^{-1}(\hat{G})$, the space
$D\phi_0(x) \circ \mathrm{i}_x^z \cdot E^1_0$ is close to $\hat{E}^1_0$,
while $D\phi_n (g^n x) \circ \mathrm{i}_{g^n x}^{f^n z} \cdot E^2_n$ is close to $\hat{E}^2_n$.
Then the second condition in the statement of the lemma follows.
\end{proof}

\subsection[A special split sequence]{A Special Split Sequence}\label{ss.special seq}

Let us now focus on some specific split sequences that come from the Oseledets splitting.

\smallskip

Given $f \in \Diff^1_\omega(M)$ and $p \in \{1, \ldots, N\}$,
we define the invariant set
\begin{multline*}
\Sigma_p(f) = \big\{ z\in M ; \;
\text{$z$ is non-periodic, Oseledets regular,}\\
\text{and $\lambda_p(f,z) > \lambda_{p+1}(f,z)$} \big\}.
\end{multline*}
We consider the splitting
\begin{equation}\label{e.splitting}
T_{\Sigma_p(f)} M = E^u \oplus E^c \oplus E^s
\end{equation}
such that at each point $E^u$, $E^c$, and $E^s$
are the sum of the Oseledets spaces corresponding respectively to the sets of
Lyapunov exponents
$$
\{\lambda_1, \ldots, \lambda_p\}, \
\{\lambda_{p+1}, \ldots, \lambda_{2N-p} = -\lambda_{p+1}\}, \text{ and }
\{\lambda_{2N - p + 1} = -\lambda_p, \ldots, \lambda_{2N} = -\lambda_1\}.
$$
We also define bundles $E^{uc}$, $E^{us}$, $E^{cs}$ respectively
as $E^u \oplus E^c$ etc.

Two obvious remarks:
First, when we speak of $E^u$, $E^c$, $E^s$, the number $p$ is implicitly fixed.
Second, despite the notation, the splitting~\eqref{e.splitting} has no reason to be partially hyperbolic.

The splitting \eqref{e.splitting} has the following properties:
\begin{gather}
\text{$Df$-invariance: $Df(z) \cdot E^* (z) = E^*(f(z))$, $* = u$, $c$, $s$}  \label{e.condition splitting 1} \\
\dim E^u = \dim E^s = p, \quad \dim E^c = 2 (N-p)  \label{e.condition splitting 2}\\
\omega( E^{u} , E^{uc}) \equiv 0, \quad
\omega( E^{c} , E^{us}) \equiv 0, \quad
\omega( E^{s} , E^{cs}) \equiv 0. \label{e.condition splitting 3}
\end{gather}
The first two are completely obvious, while \eqref{e.condition splitting 3}
follows from the fact that if $v_i$, $v_j \in T_x M$ are vectors
with respective Lyapunov exponents $\lambda_i$, $\lambda_j$ such that $\lambda_i + \lambda_j \neq 0$
then $\omega (v_i, v_j) = 0$.

\smallskip

The split sequences on $TM$ that we will be interested in are those that come
from the splitting $E^u \oplus E^{cs}$, that is, those of the form
$$
\left\{ E^u (f^i z) \oplus E^{cs} (f^i z) \xrightarrow{Df(f^i z)}
E^u (f^{i+1} z) \oplus E^{cs} (f^{i+1} z) \right\}_{0\le i < m}
$$
where $z \in \Sigma_p(f)$.
To avoid such a cumbersome notation, we write the sequence as
$Df(f^i z)\colon E^u \oplus E^{cs} \hookleftarrow$ ($0 \le i < m$).

\subsection[The Main Lemma: lack of dominance implies flexibility]
{The Main Lemma: Lack of Dominance Implies Flexibility}

If the splitting $E^u \oplus E^{cs}$ is dominated over the orbit of a point $z$,
then, due to the existence of a strictly invariant cone field,
no split sequence
$Df(f^i z)\colon E^u \oplus E^{cs} \hookleftarrow$ ($0 \le i < m$)
can be $(\eps,\kappa)$-flexible, provided $\eps>0$ is small enough.
A major part of this paper is devoted to proving the following converse to this fact:

\begin{main}
Given $f\in \Diff_\omega^1(M)$,
$\eps>0$, $\kappa>0$, and $p \in \{1,\ldots, N\}$,
there exist $m_1 \in \N$ with the following properties:

If $z \in \Sigma_p(f)$ and $m\in\N$ are such that $m \ge m_1$ and
\begin{equation}\label{e.nondominated}
\frac{\left\|Df^m (z) | E^{cs}(z) \right\|}{\m \left( Df^m (z) | E^u(z) \right)} \ge \frac{1}{2} \, ,
\end{equation}
then the split sequence 
$Df(f^i z): E^u \oplus E^{cs} \hookleftarrow$ ($0 \le i < m$)
is $(\eps, \kappa)$-flexible.
\end{main}

That is, lack of dominance expressed by \eqref{e.nondominated} implies flexibility.

\begin{rem}
In addition to \eqref{e.nondominated}, the only properties
about the splitting $E^u \oplus E^c \oplus E^s$ that we are going to use in the proof of the Main Lemma
are \eqref{e.condition splitting 1}, \eqref{e.condition splitting 2}, and \eqref{e.condition splitting 3}.
\end{rem}

The proof of the Main Lemma will occupy Sections~\ref{s.flex I} and \ref{s.flex II}.

\section[The four types of non-dominance]{The Four Types of Non-dominance} \label{s.flex I}

The aim of this section is to prove Lemma~\ref{l.4 types} below.
That proposition classifies the split sequences considered in the
Main Lemma in four types.
Each of these four types of sequences will be shown to be flexible in Section~\ref{s.flex II},
and this will prove the Main Lemma.

For the rest of this section, let $f\in \Diff_\omega(M)$ and $p \in \{1, \ldots , N\}$ be fixed.
Recall from \S\ref{ss.special seq}
the definition of the set $\Sigma_p(f)$ and the
splitting $T_{\Sigma_p(f)} M = E^u \oplus E^c \oplus E^s$.

\subsection[The classification]{The Classification}

A set of the form $\{f^i z; \; 0\le i < n\}$, where
$z \in \Sigma_p(f)$ and $n \in \N$, will be called a \emph{segment of length $n$}.

A segment $\{z, \ldots, f^{n-1} z\}$
is called \emph{of type II} (\emph{with constant $\KII>1$}) if
$$\frac{\|Df^n | E^{cs} (z) \|}{\m(Df^n | E^u (z))} > \KII.$$

A segment $\{z, \ldots, f^{n-1} z\}$
is called \emph{of type III} (\emph{with constant $\KIII>1$})
if for $0 \le i \le n$ there exist symplectic linear maps $\cL_i \colon  T_{f^i z} M \to \R^{2N}$
(that is, that send $\omega$ to the standard symplectic form $\sum_i dp_i \wedge dq_i$ on $\R^{2N}$)
such that:
\begin{itemize}
\item $\|\cL_i^{\pm 1}\| \le \KIII$.
\item The images by $\cL_i^{-1}$ of the vectors $\vetor{p_1}$ and $\vetor{q_1}$ are contained respectively in the spaces $E^u (f^i z)$ and $E^s (f^i z)$.
\item The (symplectic linear) map
    $A_i = \cL_{i+1} \circ Df(f^i z) \circ \cL_i^{-1}\colon  \R^{2N} \to \R^{2N}$
    is the identity on the $2$-plane $p_1 q_1$.
\end{itemize}

A segment $\{z, \ldots, f^{n-1} z\}$
is called \emph{of type IV} (\emph{with constants $\KIV>1$, $\tau>1$})
if there exist symplectic linear maps $\cL_i \colon  T_{z_i} M \to \R^{2N}$, $0 \le i \le n-1$,
such that:
\begin{itemize}
\item $\|\cL_i^{\pm 1}\| \le \KIV$.
\item The images by $\cL_i^{-1}$ of the vectors $\vetor{p_1}$, $\vetor{p_2}$, $\vetor{q_1}$, and $\vetor{q_2}$ are contained respectively in the spaces $E^u$, $E^c$, $E^c$, and $E^s$.
\item The (symplectic linear) map
    $A_i = \cL_{i+1} \circ Df(z_i) \circ \cL_i^{-1}\colon  \R^{2N} \to \R^{2N}$
    preserves the $4$-plane $p_1 p_2 q_1 q_2$,
    where it is given by
    $$
    A_i \colon  (p_1, p_2,q_1, q_2) \mapsto (c_i p_1, c_i p_2, c_i^{-1} q_1, c_i^{-1} q_2), \quad
    \text{where $c_i> \tau$.}
    $$
\end{itemize}

Notice that segments of type IV do not exist if $p=N$,
because in that case $E^c=\{0\}$.
(That is why type IV does not appear in \cite{BV Annals}.)

\smallskip

Recall that the \emph{symplectic complement}
of a vector space $E$ is the space $E^\omega$ formed by vectors $w$ such that $\omega(v,w)=0$ for all
$v \in E$. If $L$ is a symplectic linear map then $(L(E))^\omega = L(E^\omega)$.
It follows that if $A_i$ is the linear map as in the definition of type III (resp.~IV) then
$A_i$ preserves the $(2N-2)$-plane $p_2 \cdots p_{2N} q_2 \cdots q_{2N}$
(resp.\ the $(2N-4)$-plane $p_3 \cdots p_{2N} q_3 \cdots q_{2N}$).

\begin{lemma}\label{l.4 types}
Let $\alpha > 0$, $\KII>1$, $m_0 \in \N$.
Then there exist numbers $\KIII$, $\KIV > 1$, $\tau >1$,
where $\KIII$ does not depend on $m_0$,
with the following properties:
Assume that $z \in \Sigma_p(f)$ and $m \ge m_0$ are such that
the non-dominance condition~\eqref{e.nondominated} is satisfied.
Then one of the following holds:
\begin{enumerate}
\item[I.] There exists $i$ with $0 \le i \le m$, such that
$\angle (E^u (f^i z), E^{cs} (f^i z)) < \alpha$.

\item[II.] There exist $i$ and $j$ with $0 \le i < j \le m$
such that the segment $\{f^i z, \ldots, f^j z\}$
is of type~II with constant $\KII$.

\item[III.] There is some $i$ with $0 \le i \le m-m_0$ such that the
segment $\{f^i z, \ldots, f^{i+m_0} z\}$
is of type~III with constant $\KIII$.

\item[IV.] The segment $\{z, \ldots, f^m z\}$ is of type~IV with constants $\KIV$, $\tau$.
\end{enumerate}
\end{lemma}

\subsection{Proof}
We start with some generalities about symplectic and Riemannian structures on the manifold.

\smallskip

For each $x\in M$, let $\cJ_x\colon  T_x M \to T_x M$
be the isomorphism defined by $\omega(v,w)=\langle \cJ_x v,w \rangle$ for all
$v,w \in T_x M$.
Observe that the symplectic complement of a subspace $E \subset T_x M$ is
$E^\omega = (\cJ_x (E))^\perp$.

Denote
$$
K_\omega = \sup_{x \in M} \| \cJ_x^{\pm 1} \| \, .
$$
In particular, we have   
\begin{equation} \label{e.K omega}
|\omega (v,w)|  \leq K_\omega \|v\| \, \|w\| \quad
\text{for all $v,w \in T_x M$.}
\end{equation}

\begin{lemma}\label{l.jota}
There are functions $\beta_1(B)>0$ and $B_1(\beta)>1$ with the following properties.

Let $x \in M$, and let $E$, $F \subset T_x M$ be vector spaces with the same dimension,
and such that $E^\omega \cap F = \{0\}$.

If $\angle(E^\omega, F) > \beta >0$   
then setting $B = B_1(\beta)$ we have that
\begin{equation}\label{e.iso}
\text{$\exists$ isomorphism $J\colon  E \to F$ s.t.~}
\left\{
\begin{array}{l}
\|J^{\pm 1}\| \le B \\
|\omega (v, J(v))| \ge B^{-1} \|v\|^2
\end{array}
\right.
\end{equation}
Conversely, if~\eqref{e.iso} holds for some $B>1$ then 
$\angle (E^\omega, F) > \beta_1(B)$.
\end{lemma}

\begin{proof}
Assume that $\angle(E^\omega, F)>\beta$.
Let $p\colon  T_x M \to F$ be the projection parallel to $E^\omega$;
then $\|p\| < 1/\sin\beta$.
Let $J$ be the restriction of  $p \circ \cJ_x$ to $E$.
If $v \in E$ then
$|\omega (v, J(v))| = |\omega (v, \cJ_x(v))| = \|\cJ_x(v)\|^2  \ge K_\omega^{-2} \|v\|^2$.
Since $E^\omega = (\cJ_x(E))^\perp$, we have $\|J (v)\| \ge \|\cJ_x (v)\| \ge K_\omega^{-1} \|v\|$.
Therefore~\eqref{e.iso} holds for some appropriate $B = B_1(\beta)$.

On the other hand, if~\eqref{e.iso} holds then for any
unit vectors $v \in E^\omega$, $w\in F$ we have
$
|\omega(w-v, J^{-1}(w))| = |\omega(w, J^{-1}(w))| \ge B^{-1} \|J^{-1}(w)\|^2 \ge B^{-3}.
$
Using \eqref{e.K omega} we find a lower bound for $\|w-v\|$.
This shows that $\angle(E^\omega, F)$ is bigger than some $\beta_1(B)>0$.
\end{proof}

It follows from the lemma that   
there is a function ${\beta_2(\beta)>0}$ such that
\begin{equation}\label{e.betas}
\angle (E^\omega, F) > \beta \ \Rightarrow \ \angle(E, F^\omega) > \beta_2(\beta)
\end{equation}
(where $E$, $F \subset T_x M$ have the same dimension).

\smallskip

An (ordered) set $\{\mathbf{e}_1,\ldots, \mathbf{e}_\nu, \mathbf{f}_1, \ldots, \mathbf{f}_\nu\} \subset T_x M$
will be called \emph{orthosymplectic} if
$$
\omega(\mathbf{e}_i, \mathbf{e}_j) = \omega(\mathbf{f}_i, \mathbf{f}_j)=0, \quad
\omega(\mathbf{e}_i, \mathbf{f}_j) = \delta_{ij} \quad \text{for all $i$, $j$.}
$$
If $\nu = N$ then the set is called a \emph{symplectic basis} of $T_x M$.

\begin{lemma}\label{l.sympl basis}
For every $K_1>0$ there exist $K_2$, $K_3>0$ with the following properties.
Every orthosymplectic set
$\{\mathbf{e}_1,\ldots, \mathbf{e}_\nu, \mathbf{f}_1, \ldots, \mathbf{f}_\nu\} \subset T_x M$
such that
$$
\|\mathbf{e}_i\|, \ \|\mathbf{f}_i\| \le K_1 \quad\text{for $1 \le i \le \nu$}
$$
can be extended to a symplectic basis
$\{\mathbf{e}_1,\ldots, \mathbf{e}_N, \mathbf{f}_1, \ldots, \mathbf{f}_N\}$ such that
\begin{equation}\label{e.extended}
\|\mathbf{e}_i\|, \ \|\mathbf{f}_i\| \le K_2 \quad\text{for $\nu < i \le N$.}
\end{equation}
Furthermore, if $\cL\colon  T_x M \to \R^{2N}$ is the linear map that takes this basis
to the canonical symplectic basis
$\big\{ \vetor{p_1}, \ldots, \vetor{p_N}, \vetor{q_1}, \ldots, \vetor{q_N} \big\}$
of $\R^{2N}$ then $\|\cL^{\pm 1}\| \le K_3$.
\end{lemma}

\begin{proof}
Fix an orthosymplectic set
$\{\mathbf{e}_1,\ldots, \mathbf{e}_\nu, \mathbf{f}_1, \ldots, \mathbf{f}_\nu\} \subset T_x M$
composed of vectors of norm at most $K_1$.
Let $Y$ be the spanned space; it is a symplectic space (that is, $Y \cap Y^\omega = \{0\}$)
of dimension $2\nu$.
Let $P\colon  T_x M \to Y$ be the projection onto $Y$ parallel to $Y^\omega$.
It is given by the formula:
$$
P(v) = \sum_{i=1}^\nu \left[\omega(v, \mathbf{f}_i) \mathbf{e}_i - \omega(v, \mathbf{e}_i) \mathbf{f}_i \right] \, .
$$
By \eqref{e.K omega}, $\|P\| \le K_\omega K_1^2$.
Now assume $\nu< N$ and let us see how to extend the orthosymplectic set.
Take a unit vector $\hat{\mathbf{e}}$ orthogonal to $Y$,
and let $\mathbf{e}_{\nu+1} = \hat{\mathbf{e}} - P(\hat{\mathbf{e}})$.
Then $\mathbf{e}_{\nu+1}$ belongs to $Y^\omega$, and by Pythagoras'  Theorem,
its norm is at least $1$.
Consider the vector $\hat{\mathbf{f}} = \cJ_x(\mathbf{e}_{\nu+1}) / \|\cJ_x(\mathbf{e}_{\nu+1})\|^2$;
its norm is at most $K_\omega$, and
$\omega(\mathbf{e}_{\nu+1}, \hat{\mathbf{f}}) = 1$.
Let $\mathbf{f}_{\nu+1} = \hat{\mathbf{f}} - P(\hat{\mathbf{f}})$.
Then $\mathbf{f}_{\nu+1}$ belongs to $Y^\omega$ and
$\omega(\mathbf{e}_{\nu+1}, \mathbf{f}_{\nu+1}) = 1$,
so the enlarged set $\{\mathbf{e}_1,\ldots, \mathbf{e}_{\nu+1}, \mathbf{f}_1, \ldots, \mathbf{f}_{\nu+1}\}$
is orthosymplectic.
Also, we can bound $\|\mathbf{e}_{\nu+1}\|$ and $\|\mathbf{f}_{\nu+1}\|$ by functions of $K_1$.
Continuing by induction, we find the desired symplectic basis.

Now let $\cL$ be as in the statement of the lemma.
Obviously an upper bound for $\|\cL^{-1}\|$ can be found using \eqref{e.extended}.
On the other hand,
if $\cL(v) = (p_1, \ldots, p_N, q_1, \ldots q_N)$
then $p_i = \omega(v, \mathbf{f}_i)$ and $q_i = -\omega(v,\mathbf{e}_i)$.
So we can bound $\|\cL\|$ as well.
\end{proof}

\smallskip

Let us adopt the following notation:
If $A$ and $B$ are positive quantities then
$$
A \lesssim B \pmod{a, b, \ldots}
$$
means that $B/A$ is bigger than some positive quantity depending only on $a$, $b$, \ldots
(and maybe on $M$, $f$, and $p$, which are fixed).
Then $A \approx B$ and $A \gtrsim B \pmod{a,b,\ldots}$ are defined in the obvious ways.

\smallskip

Now we are ready to give the:

\begin{proof}[Proof of Lemma~\ref{l.4 types}]
Let $\alpha$, $K = \KII$, $m_0$ be given.
Let $z$ belong to $\Sigma_p(f)$, and let $z_i = f^i z$.
Assume that for some $m \ge m_0$,
the segment $\{z_0, \ldots, z_m\}$ is non-dominated, meaning that \eqref{e.nondominated} holds.

From now on, assume that
\begin{equation}\label{e.not I}
\angle (E^u_i, E^{cs}_i) \ge \alpha, \quad \text{for every $i$ with $0 \le i \le m$}
\end{equation}
and
\begin{equation}\label{e.not II}
\frac{\|Df^n | E^{cs}_i \|}{\m(Df^n | E^u_i)} \le K ,  \quad \text{for every $i$, $n$ with $0 \le i < i+n \le m$,}
\end{equation}
because otherwise we fall in one of the first two cases and there is nothing to prove.

We claim that:
\begin{equation}\label{e.improved not I}
\angle (E^u_i, E^{cs}_i), \
\angle (E^c_i, E^{us}_i), \
\angle (E^s_i, E^{uc}_i)  \gtrsim 1 \pmod{\alpha},  \quad \text{for every $i$.}
\end{equation}
From~\eqref{e.condition splitting 3} we see that
$E^{cs} = (E^s)^\omega$ and $(E^u)^\omega = E^{uc}$.
So using \eqref{e.not I} and \eqref{e.betas} we get that
$\angle (E^s_i, E^{uc}_i) >\beta_2(\alpha)$.
So we got two bounds in \eqref{e.improved not I}, and the third follows
(use for instance Lemma~2.6 from \cite{BV Annals}.)

\begin{subl}\label{sl.sympl symmetry}
Let $(\prime,\prime \prime)$ be either $(u,s)$, $(c,c)$, or $(s,u)$.
Let $i \in \{0,\ldots,m\}$.
\begin{enumerate}
\item[1.] For every unit vector $v$ in $E'_i$, there exists a unit vector $v^\star$ in $E''_i$
    such that
    $$
    |\omega (v, v^\star)| \gtrsim 1 \pmod{\alpha}.
    $$
\end{enumerate}
Moreover, if $n \in \Z$ is such that $i+n \in \{0,\ldots,m\}$ then:
\begin{enumerate}
\item[2.] If $v \in E_i'$ is a unit vector then
    $\|Df^ n (v) \| \, \|Df^n (v^\star)\| \gtrsim 1 \pmod{\alpha}$.
\item[3.] $\m(Df^n|E'_i) \, \|Df^n|E''_i\| \approx 1 \pmod{\alpha}$.
\item[4.] If $v$ is a unit vector in $E'_i$ such that $\| Df^n v \| = \m(Df^n|E_i')$ then
    $$
    \|Df^n (v^\star)\| \approx \|Df^n|E_i''\| \pmod{\alpha} .
    $$
    (That is, if $v$ is the unit vector that is most contracted by $Df^n|E'_i$,
    then $v^\star$ is a unit vector that is almost-the-most expanded by $Df^n|E''_i$.)
\end{enumerate}
\end{subl}

\begin{proof}
Let $\prime$, $\prime \prime$, $i$, $n$ be as in the statement.
By \eqref{e.improved not I},
$\angle ((E'_i)^\omega, E''_i)  \gtrsim 1 \pmod{\alpha}$.
Let $J_i'\colon E_i' \to E_i''$ be given by Lemma~\ref{l.jota}.
If $v \in E_i'$ is a unit vector, let $v^\star = J'_i(v) / \|J'_i v\|$.
Then $v^\star$ has the properties as in item 1.
Item 2 is evident:
$$
K_\omega \|Df^n (v)\| \, \|Df^n (v^\star)\| \ge
|\omega( Df^n (v), Df^n (v^\star))| =
|\omega(v, v^\star)| \gtrsim 1 \pmod{\alpha}.
$$
Now let $v$ be a unit vector in $E'_i$ such that $\| Df^n v \| = \m(Df^n|E_i')$.
By item 2,
$$
\m(Df^n|E'_i) \, \|Df^n|E''_i\| \ge \|Df^n (v)\| \, \|Df^n (v^\star)\|  \gtrsim 1 \pmod{\alpha} ,
$$
proving one inequality in item 3.
The other inequality
follows from the first, replacing $(i,n)$ by $(i+n,-n)$.
Item 4 follows from items 2 and~3:
$$
\|Df^n|E_i''\| \ge \|Df^n (v^\star)\| \gtrsim \frac{1}{\|Df^n (v)\|} = \frac{1}{\m(Df^n|E'_i)} \approx \|Df^n|E_i''\|
\pmod{\alpha}. \qedhere
$$
\end{proof}

Now we extract consequences from~\eqref{e.not II}:
\begin{subl}\label{sl.improved non II}
For any $i$, $n$ with $0 \le i < i+n \le m$, we have
$$
\xymatrix@=0pt{\|Df^n|E^s_i \| \ar@{-}@<-2pt> `d[r] `[rrrrrrrr]  &\lesssim
            &\m(Df^n|E^c_i) \ar@{-}@<+2pt> `d[r] `[rrrr]
            &\lesssim &1 &\lesssim
            &\|Df^n|E^c_i\| &\lesssim &\m(Df^n|E^u_i)}
\pmod{\alpha, K}
$$
Moreover, the matched pairs have product $\approx 1 \pmod{\alpha,K}$.
\end{subl}

\begin{proof}
By~\eqref{e.not II},
$$
\m(Df^n | E^u_i) \gtrsim \|Df^n | E^{cs}_i\| \ge \|Df^n | E^c_i\| \pmod{K}.
$$
Then the other assertions follow easily from Sublemma~\ref{sl.sympl symmetry} (item~3).
\end{proof}

\begin{subl}\label{sl.case III}
If
\begin{equation}\label{e.III}
\frac{\big\|Df^{m_0} | E^s_k \big\|}{\m\big(Df^{m_0} | E^u_k\big)} \ge \frac{1}{2} \, , \quad
\text{for some $k$ with $0 \le k \le m-m_0$}
\end{equation}
then the segment $\{z_k, \ldots, z_{k+m_0}\}$ is of type III
(with some constant $\KIII$ that depends only on $\alpha$ and $K = \KII$).
\end{subl}

The interpretation of \eqref{e.III} is that the segment $\{z_k, \ldots, z_{k+m_0}\}$
is non-dominated in a stronger way: $E^u$ does not dominate $E^s$.

\begin{proof}
Together with Sublemma~\ref{sl.improved non II}, the assumption~\eqref{e.III} gives
$$
\|Df^{m_0} | E^s_k\| \approx 1 \approx \m(Df^{m_0} | E^u_k) \pmod{\alpha,K}.
$$

Let $v$ be a unit vector in $E_k^u$ that is least expanded by $Df^{m_0}$,
that is $\|Df^{m_0} v\| = \m(Df^{m_0} | E_k^u)$.
By Sublemma~\ref{sl.sympl symmetry}, the unit vector
$v^\star \in E_k^s$ satisfies
$\|Df^{m_0} (v^\star)\| \approx \| Df^{m_0} | E_k^s\| \pmod{\alpha}$.
Using~\eqref{e.not II} we get, for each $i = 0, \ldots, m_0$,
$$
K \ge
\frac{\|Df^i (v^\star)\|}{\|Df^i (v)\|} \ge
\frac{\|Df^{m_0} (v^\star)\| \, / \, \|Df^{m_0-i}| E^s_{k+i} \|}{\|Df^{m_0} (v)\| \, / \, \m(Df^{m_0-i}| E^u_{k+i})}
\gtrsim 1 \pmod{\alpha,K}.
$$
That is, $\|Df^i (v^\star)\| \approx \|Df^i (v)\|$.
In addition, both norms are $\approx 1$, by Sublemma~\ref{sl.improved non II}.
For each $i = 0, \ldots, m_0$, let
$$
\mathbf{e}_{1,i} = Df^i (v), \quad
\mathbf{f}_{1,i} = \frac{Df^i (v^\star)}{\omega(v, v^\star)} \, .
$$
Then $\{\mathbf{e}_{1,i}, \mathbf{f}_{1,i}\}$ is a orthosymplectic subset of $T_{z_{k+i}} M$.
By Lemma~\ref{l.sympl basis}, we can extend it to a symplectic basis
$\{\mathbf{e}_{1,i}, \mathbf{f}_{1,i}, \ldots, \mathbf{e}_{N,i}, \mathbf{f}_{N,i}\}$,
and furthermore if $\cL_i$ is the linear map that takes this basis to the
canonical symplectic basis of $\R^{2N}$ then
$\|\cL_i^{\pm 1}\| \lesssim 1 \pmod{\alpha,K}$.
The map $\cL_{i+1} \circ Df(z_{k+i}) \circ \cL_i^{-1}$
is the identity on the plane $p_1 q_1$.
This shows that the segment being considered is of type~III.
\end{proof}

Sublemma~\ref{sl.case III} says that if \eqref{e.III} holds then we are done.
Assume from now on that \eqref{e.III} does not hold, that is,
\begin{equation}\label{e.not III}
\frac{\|Df^{m_0} | E^s_k\|}{\m(Df^{m_0} | E^u_k)} < \frac{1}{2} \, , \quad
\text{for all $k$ with $0 \le k \le m-m_0$}
\end{equation}
\emph{From now on, all relations $\gtrsim$, $\lesssim$, $\approx$ will be meant mod $\alpha$, $K$, $m_0$.}

\begin{subl}
$E^u$ is uniformly expanding and $E^s$ is uniformly contracting.
That is, there exists $\lambda > 1$ and $C>1$ (depending on $\alpha$, $K$, $m_0$)
such that
\begin{equation}\label{e.uniform u and s}
\left.
\begin{array}{l}
\m(Df^n | E^u_i ) > C^{-1} \lambda^n \\
\|Df^n | E^s_i \| < C \lambda^{-n}
\end{array}
\right\}
\text{$\forall$ $i$, $n$ with $0 \le i < i+n \le m$.}
\end{equation}
\end{subl}

\begin{proof}
It follows from~\eqref{e.not III} that
$$
\frac{\|Df^n | E^s_i \|}{\m(Df^n | E^u_i )} \le
\left(\sup_{x\in M} \frac{\|Df(x)\|}{\m(Df(x))} \right)^{m_0 - 1}
\left(\frac{1}{2}\right)^{\left\lfloor \frac{n}{m_0}\right\rfloor} \, .
$$
The right hand side is exponentially small with $n$.
Since $\|Df^n | E^s_i \| \approx 1 / \m(DF^n | E^u_i )$,
the lemma follows.
\end{proof}

For the first time, let us use the hypothesis of non-domination
of the segment $\{z_0, \ldots, z_m\}$:
\begin{equation}\label{e.nondominated again}
\frac{\|Df^m | E^{cs}_0 \|}{\m( Df^m | E^u_0)} \ge \frac{1}{2} \, .
\end{equation}

We claim that:
\begin{equation}\label{e.improved nondominated}
\|Df^m | E^s_0\|  \approx \m(Df^m | E^c_0) \quad \text{and} \quad
\|Df^m | E^c_0\|  \approx \m(Df^m | E^u_0).
\end{equation}
Since $\|Df^m | E^s_0\| \lesssim 1 \lesssim \|Df^m | E^c_0 \|$ and $\angle(E^s_0, E^c_0) \approx 1$,
we have $\|Df^m | E^{cs}_0 \| \approx \|Df^m | E^c_0 \|$.
So~\eqref{e.nondominated again}, together with Sublemma~\ref{sl.improved non II},
gives the second relation in~\eqref{e.improved nondominated}.
The first relation follows from the second.

Let $v^u \in E^u_0$ and $v^{cs} \in E^c_0$ be unit vectors such that
$$
\|Df^m v^u\| = \m(Df^m | E^u_0) \quad \text{and} \quad
\|Df^m v^{cs} \| = \m(Df^m | E^c_0).
$$
Let $v^s = (v^u)^\star \in E^s_0$ and $v^{cu} = (v^{cs})^\star \in E^c_0$.
Then, by Sublemma~\ref{sl.sympl symmetry},
$$
\|Df^m v^s\|  \approx \|Df^m | E_0^u \| \quad \text{and} \quad
\|Df^m v^{cu} \| \approx  \| Df^m | E^c_0 \|.
$$

\begin{subl}
If $0\le i \le m$ then
\begin{equation}\label{e.4 approx}
\|Df^i v^u\| \approx \|Df^i v^{cu}\| \approx \frac{1}{\|Df^i v^{cs}\|} \approx \frac{1}{\|Df^i v^s \|} \, .
\end{equation}
If $n>0$ and $i+n \le m$ then
\begin{equation}\label{e.bla}
\m(Df^n | E^u_i) \approx \frac{\|Df^{n+i} v^u\|}{\|Df^i v^u\|} \, .
\end{equation}
\end{subl}

\begin{proof}
From~\eqref{e.improved nondominated},
$\|Df^m  v^u\| \approx \|Df^m v^{cu} \|$.
Therefore
$$
K \ge \frac{\|Df^i v^{cu}\|}{\|Df^i v^u\|} \ge
\frac{\|Df^m v^{cu}\|/ \|Df^{m-i}| E^c_i \|}{\|Df^m v^u\| / \m(Df^{m-i}| E^u_0)} \gtrsim 1,
$$
that is, $\|Df^i v^u\| \approx \|Df^i v^{cu} \|$.
Analogously, $\|Df^i v^s\| \approx \|Df^i v^{cs}\|$.
Now, from Sublemma~\ref{sl.improved non II},
$$
\m(Df^n|E^u_i) \le
\frac{\|Df^{n+i} v^u\|}{\|Df^i v^u\|} \approx
\frac{\|Df^{n+i} v^{cu}\|}{\|Df^i v^{cu}\|} \le
\|Df^n| E^c_i \| \lesssim
\m(Df^n|E^u_i),
$$
proving~\eqref{e.bla}.
In particular, $\|Df^n v^u\| \approx \m(Df^n |E^u_0)$. 
Analogously, $\|Df^n v^s\| \approx \| Df^n |E^s_0 \|$.
Therefore $\|Df^n v^u\| \approx 1 / \|Df^n v^s\|$,
completing the proof of~\eqref{e.4 approx}.
\end{proof}

For $i=0,\ldots,m$, let
\begin{alignat*}{2}
\mathbf{e}_{1,i} &= \frac{Df^i v^u}{\|Df^i v^u\|}                           \, , &\quad
\mathbf{f}_{1,i} &= \frac{\|Df^i v^u\| \, Df^i v^s}{\omega(v^u, v^s)}\, , \\
\mathbf{e}_{2,i} &= \frac{Df^i v^{cu}}{\|Df^i v^u\|}                           \, , &\quad
\mathbf{e}_{2,i} &= \frac{\|Df^i v^u\| \, Df^i v^{cs}}{\omega(v^{cu}, v^{cs})}\, .
\end{alignat*}
Then $\{\mathbf{e}_{1,i}, \mathbf{f}_{1,i}, \mathbf{e}_{2,i}, \mathbf{f}_{2,i}\}$
is a orthosymplectic subset of $T_{z_i} M$.
By Lemma~\ref{l.sympl basis}, we can extend it to a symplectic basis
$\{\mathbf{e}_{1,i}, \mathbf{f}_{1,i}, \ldots, \mathbf{e}_{N,i}, \mathbf{f}_{N,i}\}$,
and furthermore if $\cL_i$ is the linear map that takes this basis to the
canonical symplectic basis of $\R^{2N}$ then
$\|\cL_i^{\pm 1}\| \lesssim 1$.
The restriction of the map $A_i = \cL_{i+1} \circ Df(z_i) \circ \cL_i^{-1}$
to the $4$-plane $p_1 p_2 q_1 q_2$ is given by
$$
A_i \colon  (p_1, p_2,q_1, q_2) \mapsto (c_i p_1, c_i p_2, c_i^{-1} q_1, c_i^{-1} q_2)
\quad \text{where } c_i =  \frac{\|Df^{i+1} v^u\|}{\|Df^i v^u\|} \, .
$$
Unfortunately, $c_i$ is not necessarily always bigger than $1$ as required in the definition of type IV.
To remedy that:

\begin{subl}
Given $C_1>0$, $\delta_1>0$, and $\ell \in \N$ there exist $C_2>0$, $\delta_2>0$
with the following properties:
Given a sequence $\{a_i\}_{i=0}^{m-1}$ with $|a_i| \le C_1$ for each $i$ and
$\sum_{j=i}^{i+\ell-1} a_j > \delta_1$
for $0 \le i \le m-\ell$,
there exists a sequence $\{b_i\}_{i=0}^m$ such that $|b_i| \le C_2$ and
$b_{i+1} + a_i - b_i > \delta_2$ for each $i$.
\end{subl}

\begin{proof}
Let $a_i = C_1$ for $i \ge m$.
Let $b_i = \frac{1}{\ell} \sum_{j=0}^{\ell-1} (\ell-1-j) a_{i+j}$.
Then $b_{i+1}+a_i-b_i = \frac{1}{\ell} \sum_{j=0}^{\ell-1} a_{i+j} > \delta_1 / \ell$.
\end{proof}

Let $a_i = \log c_i$ and let $b_i$ be given by the sublemma.
Let $D_i \colon  \R^{2N} \to \R^{2N}$ be the symplectic linear map
defined by $D_i(\vetor{p_j}) = e^{b_i} \vetor{p_j}$, $D_i(\vetor{q_j}) = e^{-b_i} \vetor{q_j}$.
Consider the new map $\hat{\cL}_i = D_i \circ \cL_i$;
then the action of
$\hat{\cL}_{i+1} \circ Df(z_i) \circ \hat{\cL}_i^{-1}$
on the $4$-plane $p_1 p_2 q_1 q_2$ is given by
$$
(p_1, p_2,q_1, q_2) \mapsto (\hat{c}_i p_1, \hat{c}_i p_2, \hat{c}_i^{-1} q_1, \hat{c}_i^{-1} q_2)
\quad \text{where } \hat{c}_i = e^{b_{i+1}-b_i} c_i \, .
$$
We have $\hat{c}_i > \tau > 1$ where $\tau$ depends only on $\alpha$, $K$, and $m_0$.
This proves that the segment $\{z_0,\ldots,z_m\}$ is of type~IV,
completing the proof of Lemma~\ref{l.4 types}.
\end{proof}

\section[Proof of flexibility]{Proof of Flexibility} \label{s.flex II}

The goal of this section is to prove the Main Lemma.
Thus we will show that each of the cases I--IV from Lemma~\ref{l.4 types}
implies flexibility.

Let the diffeomorphism $f$, $p \in \{1,\ldots,N\}$, $\eps>0$, and $\kappa>0$
be fixed throughout this section.
For concision, we will say that a
segment $\{z, \ldots, f^{n-1} z\}$ (with $z\in \Sigma_p(f)$) is \emph{flexible}
if the split sequence
$Df(f^i z)\colon {E^u \oplus E^{cs} \hookleftarrow}$ ($0 \le i < n$)
is $(\eps,\kappa)$-flexible.

We now state four lemmas:

\begin{lemma}\label{l.case I}
There is $\alpha>0$ such that if $z \in \Sigma_p(f)$ satisfies
$\angle(E^u (z) , E^{cs} (z))< \alpha$
then the segment (of length $1$) $\{z\}$ is flexible.
\end{lemma}

\begin{lemma}\label{l.case II}
There is $\KII>1$
such that if a segment $\{z, \ldots, f^{n-1} z\}$, with $z\in \Sigma_p(f)$ is
of type~II with constant $\KII$ then it is flexible.
\end{lemma}

\begin{lemma}\label{l.case III}
Given $\KIII>1$, there exists $m_0$  such that
if a segment $\{z, \ldots, f^{m_0-1} z\}$
is of type~III with constant $\KIII$
then it is flexible.
\end{lemma}

\begin{lemma}\label{l.case IV}
Given $\KIV>1$ and $\tau>1$ there exists $m_1$ such that
if a segment of length $m \ge m_1$ is of type~IV with constants $\KIV$, $\tau$
then it is flexible.
\end{lemma}

Assuming Lemmas~\ref{l.case I}--\ref{l.case IV}, we can give the:

\begin{proof}[Proof of the Main Lemma]
Let $\alpha$ and $\KII$ be given by Lemmas~\ref{l.case I} and \ref{l.case II}, respectively.
Let $\KIII = \KIII(\alpha,\KII)$ be given by Lemma~\ref{l.4 types}.
Let $m_0 = m_0(\KIII)$ be given by Lemma~\ref{l.case III}.
Let $\KIV = \KIV(\alpha, \KII, m_0)$ and $\tau = \tau(\alpha, \KII, m_0)$ be given by Lemma~\ref{l.4 types}.
Finally, let $m_1 = m_1(\KIV,\tau)$ be given by Lemma~\ref{l.case IV}.
We can assume $m_1 \ge m_0$.

Now, if $m \ge m_1$ and the segment $\{z, \ldots, f^m z\}$
is non-dominated (meaning that \eqref{e.nondominated} is satisfied)
then one of the four alternatives in Lemma~\ref{l.4 types} hold.
Lemmas~\ref{l.case I}--\ref{l.case IV} imply that in each case the segment
contains a flexible subsegment.
So, by Lemma~\ref{l.short seq}, the whole segment is flexible.
\end{proof}

\subsection[Dealing with cases I and II]{Dealing with Cases I and II}

\begin{lemma}\label{l.alpha rot}
Given $\eps>0$ and $\kappa>0$, there exists $\alpha>0$ with the following properties:
If $v$, $w$ are unit vectors in $\R^{2N}$ with $\angle(v,w)<\alpha$,
and $U \subset \R^{2N}$ is a non-empty open set,
then there exists $h \in \Diff_\omega^1(\R^{2N})$ that equals the identity
outside of $U$, $\|Dh -\Id\| < \eps$ uniformly,
and such that the set $G$ of points $x \in U$ such that $Dh(x) \cdot v = w$
has measure $\mu(G) > (1-\kappa)\mu(U)$.
\end{lemma}

\begin{proof}
This follows from Lemmas~5.7 and 5.12 from~\cite{BV Annals}.
\end{proof}

\begin{proof}[Proof of Lemma~\ref{l.case I}]
It follows easily from Lemma~\ref{l.alpha rot}.
\end{proof}

\begin{proof}[Proof of Lemma~\ref{l.case II}]
It follows from Lemma~\ref{l.alpha rot} applied twice.
More precisely, one takes the unit vector in $E^u(z)$ that is least expanded by $Df^n$,
and rotates it (using Lemma~\ref{l.alpha rot}) towards the direction in $E^{cs}(z)$, 
which is most expanded by $Df^n$.
The image of the rotated vector by $Df^n$ then gets close to
$E^{cs}(f^n z)$, so with another rotation we are done.
The reader can either fill the details for himself, or else see \cite[p.~1449]{BV Annals}.
\end{proof}

\subsection[Hamiltonians and dimension reduction]{Hamiltonians and Dimension Reduction}

Let us see a procedure that
will permit us to essentially reduce the proofs of
Lemmas~\ref{l.case III} and \ref{l.case IV} to dimensions $2$ and $4$,
respectively.

For $\nu<N$, let
$$
\R^{2\nu} = \big \{(p_1,\ldots,p_N, q_1, \ldots, q_N) \in \R^{2N} ; \; p_i = q_i = 0  \text{ for } i>\nu \big\}
$$
Notice the standard symplectic form on $\R^{2N}$ restricted to $\R^{2\nu}$ coincides
with the standard symplectic form on $\R^{2\nu}$.
Also, $(\R^{2\nu})^\omega = \{ p_i = q_i = 0  \text{ for } i \le \nu \}$,
so $\R^{2N} = \R^{2\nu} \oplus (\R^{2\nu})^\omega$.
In what follows, we write
$$
\R^{2N} = \big \{(x,y) ; \; x\in \R^{2\nu}, \ y\in(\R^{2\nu})^\omega \big\} .
$$
If a symplectic map $A\colon \R^{2N} \to \R^{2N}$ preserves $\R^{2\nu}$
then it also preserves the symplectic complement $(\R^{2\nu})^\omega$,
so $A$ can be written as $A(x, y) = (B(x), C(y))$,
where $B$ and $C$ are symplectic maps on $\R^{2\nu}$ and $(\R^{2\nu})^\omega$, respectively.

If $H$ is a smooth (ie, $C^\infty$) function on $\R^{2N}$, then we let $\varphi_H^t$ denote the
Hamiltonian flow generated by $H$.

\begin{lemma} \label{l.ode}
Let $H\colon \R^{2N} \to \R$ be a smooth function that is constant outside a compact set.
Then the associated Hamiltonian flow $\varphi_H^t \colon \R^{2N} \to \R^{2N}$
is defined for every time $t\in\R$, and
$$
\| \varphi_H^t(\xi) - \xi \|   \leq |t| \sup \|DH\|, \qquad
\| D(\varphi_H^t)(\xi) - \Id \| \leq \exp \big(|t| \sup \|D^2 H\|\big)-1.
$$
for every $\xi \in \R^{2N}$ and $t\in\R$.
\end{lemma}

\begin{proof}
The last assertion follows from a Gronwall inequality applied to the Lipschitz
function $u(t) = 1 + \sup \| D \varphi_H^t - \Id \|$.
\end{proof}

\begin{lemma}\label{l.dim reduction}
Given $\nu \in \{1,\ldots,N-1\}$, $\delta>0$, $\kappa>0$, and also:
\begin{itemize}
\item symplectic linear maps $A_0$, \ldots, $A_{m-1}\colon  \R^{2N} \to \R^{2N}$ preserving $\R^{2\nu}$,
so that we can write $A_i(x,y) = (B_i(x), C_i(y))$,
for $x \in \R^{2\nu}$, $y \in (\R^{2\nu})^\omega$;
\item for each  $i=0,\ldots, m-1$, a smooth function $H_i \colon  \R^{2 \nu} \to \R$
such that $\|D^2 H_i\| < \delta$ uniformly and $H_i$ is constant outside of $B_{i-1} \ccircs B_0(U)$,
where $U$ is the open unit ball in~$\R^{2\nu}$.
\end{itemize}
Then there exist:
\begin{itemize}
\item a cylinder $\hat U = \{(x,y) \in \R^{2N}; \; \|x\|<1, \ \|y\|< a \}$, where $a>0$;
\item smooth functions $\hat H_i \colon  \R^{2N} \to \R$ such that
$\|D^2 \hat H_i\| < 2\delta$ uniformly and
$\hat H_i$ is constant outside of $A_{i-1} \ccircs A_0(\hat U)$;
\item a set $\hat{G} \subset \hat{U}$ with $\mu(\hat{G}) > (1-\kappa) \mu(\hat{U})$
such that if $(x,y) \in \hat{G}$ then
\begin{multline} \label{e.hamilton skew}
A_{m-1} \circ \varphi_{\hat H_{m-1}}^t \ccircs A_0 \circ \varphi_{\hat H_0}^t (x,y) =\\
\big(B_{m-1} \circ  \varphi_{H_{m-1}}^t \ccircs B_0 \circ \varphi_{H_0}^t (x) , \,
C_{m-1} \ccircs C_0(y) \big) \, .
\end{multline}
\end{itemize}
\end{lemma}

\begin{proof}
Let $\cB_0$, $\cC_0$ be the open unit balls in $\R^{2\nu}$, $(\R^{2\nu})^\omega$, respectively.
Let $\cB_i = B_{i-1} \ccircs B_0 (\cB_0)$,
$\cC_i = C_{i-1} \ccircs C_0 (\cC_0)$.
Let $0<\sigma<1$ be such that $\sigma^{2(N-\nu)} > 1-\kappa$.
Let $\zeta\colon  \R \to [0,1]$ be a smooth function such that:
$$
\zeta(t)=1 \text{ for } t\leq \sigma, \quad
\zeta(t)=0 \text{ for } t\geq 1, \quad
|\zeta'(t)| \leq \frac{10}{1-\sigma} \, , \quad
|\zeta''(t)| \leq \frac{10}{(1-\sigma)^2} \, .
$$
Let $a \gg 1$ (to be specified later).
Define $\psi_i\colon  (\R^{2\nu})^\omega \to \R$ by
$\psi_i (y) = \zeta \left(a^{-1} \|C_0^{-1} \cdots C_{i-1}^{-1} (y)\| \right)$.
Then
$$
\psi_i(y)=1 \text{ for } y \in \sigma a\cC_i,
\quad\text{and}\quad
\psi_i(y)=0 \text{ for } y \notin a\cC_i.
$$
Letting $c = c(\sigma)$ be an upper bound for the norms of
the first and second derivatives of the function
$y\in (\R^{2\nu})^\omega \mapsto \zeta(\|y\|)$, we can write
$$
\|D\psi_i\| \le c a^{-1} \|C_0^{-1} \cdots C_{i-1}^{-1}\|
\quad\text{and}\quad
\|D^2\psi_i\| \le c a^{-2} \|C_0^{-1} \cdots C_{i-1}^{-1}\|^2.
$$
So if $a$ is large enough, $\|D \psi_i \|$ and $\|D^2\psi_i\|$ are both uniformly small,
for every~$i$.

There is no loss in generality if we assume that each $H_i$ is zero outside~$\cB_i$.
Define $\hat{H}_i(x,y) = H_i(x) \psi_i(y)$.
Writing $v =( v_x, v_y) \in \R^{2\nu} \oplus (\R^{2\nu})^\omega$ and
analogously for $w$, we compute:
\begin{multline*}
D^2 \hat{H}_i(x,y)(v,w) = H_i(x) \cdot  D^2\psi_i(y) (v_y,w_y) +
DH_i(x)(w_x) \cdot D\psi_i(y)(v_y)  + \\
DH_i(x)(v_x) \cdot D\psi_i(y)(w_y)  + D^2 H_i(x) (v_x,w_x) \cdot \psi_i(y) .
\end{multline*}
Therefore $\| D^2 \hat H_i \| < 2\delta$ for every $i$,
provided $a$ is chosen sufficiently large.

Define the subsets of $\R^{2N}$:
$$
\hat{U} = \cB_0 \oplus (a\cC_0) \quad \text{and} \quad
\hat{G} = \cB_0 \oplus (\sigma a\cC_0) .
$$
The choice of $\sigma$ implies that $\mu(\hat{G}) >(1-\kappa) \mu(\hat{U})$.
We have $\hat H_i(x,y) = 0$ if $x\not\in \cB_i$ or $y \not\in a\cC_i$,
that is, if $(x,y) \not\in A_{i-1} \ccircs A_0(\hat{U})$.
Moreover, if
$(x,y) \in \cB_i \oplus (\sigma a \cC_i)$ then
$\varphi_{\hat H_i}^t(x,y) = \big( \varphi_{H_i}^t(x), y \big)$.
So~\eqref{e.hamilton skew} follows.
\end{proof}

In \S\ref{ss.IV} we will use the following lemma about
change of coordinates in hamiltonians.
The easy proof is left to the reader.

\begin{lemma}\label{l.hamilton change coord}
Let $H$ be a hamiltonian on $\R^{2N}$, $a>0$,
and $M\colon \R^{2N} \to \R^{2N}$ be a symplectic linear map.
Define hamiltonians $H_1(x) = a^{-2} H(ax)$ and $H_2(x) = H(M(x))$.
Then
\begin{alignat*}{2}
D^2 H_1 (x) \cdot (v,w) &= D^2 H(ax) \cdot (v,w),     &\qquad
\varphi_{H_1}^t(x)      &= a^{-1} \varphi_H^t(ax), \\
D^2 H_2(x) \cdot (v,w)  &= D^2 H(M(x))  \cdot (Mv, Mw),   &\qquad
\varphi_{H_2}^t(x)      &= M^{-1} \circ \varphi_H^t \circ M (x).
\end{alignat*}
\end{lemma}

\subsection[Dealing with case III]{Dealing with Case III}\label{ss.III}

\begin{proof}[Proof of Lemma~\ref{l.case III}]
We will assume $2N > 2$.
(The reader can adapt the arguments for the simpler $2$-dimensional case,
if he desires to reobtain the results of \cite{B GZLE}.)

Let $\KIII$ (and also $\eps$, $\kappa$) be given.
Let $\eps' = (K_\cA \KIII)^{-2} \eps$.
(Recall the definition of $K_\cA$ from \S\ref{ss.flex TM}.)
Let $\delta$ be such that $e^{2\delta}-1 = \eps'$.
Let $\sigma = 1 - \kappa/2$.
Take a smooth function $\rho\colon  \R_+ \to \R$ such that
$$
\rho(t)= t \text{ for } 0\le t \le \sigma , \quad
\rho(t)= 1 \text{ for } t \ge 1, \quad
0\le \rho'(t) \le 1, \quad
|\rho''(t)| \le \frac{10}{1-\sigma}
$$
Let $\alpha>0$ and define $H (p_1, q_1) =\frac{\alpha}{2} \rho\left(p_1^2+q_1^2 \right)$.
Notice that $\varphi_H^t$ restricted to the disk $p_1^2+q_1^2 \le \sigma$
is a rotation of angle $t\alpha$.
Choose $m$ big enough so that setting $\alpha = \frac{\pi}{2m}$
we have $\|D^2 H\|< \delta$ uniformly.
Let us  see that $m_0=m$ has the desired properties.

Take a segment $\{z,\ldots,f^{m-1}z\}$ of type~III with constant $\KIII$.
Let $\cL_i \colon  T_{f^i z} M \to \R^{2N}$
and $A_i  = \cL_{i+1} \circ Df(f^i z) \circ \cL_i$
be as in the definition of type III.
Our aim is to show that the split sequence
$Df(f^i z)\colon {E^u \oplus E^{cs} \hookleftarrow}$ ($0 \le i < m$)
is $(\eps, \kappa)$-flexible.
Because of Lemma~\ref{l.flex change coord}, it suffices to show that the split sequence on $\R^{2N}$
$$
\left\{ F_i^u \oplus F_i^{cs} \xrightarrow{A_i} F_{i+1}^u \oplus F_{i+1}^{cs} \right\}_{0\le i <m}
\quad \text{where}  \quad
F_i^* = \cL_i(E^*(f^i z))
$$
is $(\eps', \kappa)$-flexible.

The maps $A_i$ are the identity on the plane $\R^2$ spanned by
$\vetor{p_1}$ and $\vetor{q_1}$.
So we can write $A_i(x,y) = (x, C_i(y))$ for
$x \in \R^2$, $y \in (\R^2)^\omega$.
Apply Lemma~\ref{l.dim reduction}
with $\nu=2$, $H_i = H$ for $0 \le i <m$,
and $\kappa/2$ in the place of $\kappa$.
We obtain a cylinder $\hat{U}$,
hamiltonians $\hat{H}_i$ that are constant outside $A_{i-1} \ccircs A_0(\hat{U})$
and satisfy $\|D^2 \hat{H}_i\|< 2\delta$,
and a set $\hat G \subset \hat U$ with measure $> (1-\kappa/2) \mu(\hat{U})$ where
$$
A_{n-1} \circ \varphi_{\hat H_{n-1}}^1 \ccircs A_0 \circ \varphi_{\hat H_0}^1 (x,y) =
\big(\varphi_H^m (x) , \, C_{n-1} \ccircs C_0(y) \big) \, .
$$

Let $g_i = A_i \circ \varphi^1_{\hat{H}_i}$.
We check that the maps $g_i$ have the properties demanded by flexibility
(for any $\gamma>0$, in fact):
\begin{enumerate}
\item $g_i = A_i$ outside  $A_{i-1} \ccircs A_0(\hat{U})$.
\item By Lemma~\ref{l.ode}, $\|D(A_i^{-1} \circ g_i)-\Id\| < e^{2\delta}-1 = \eps'$.
\item The cylinder $\hat{U} \cap \{p_1^2+q_1^2 < \sigma\}$ has measure $\sigma \mu(\hat{U})$;
let $G$ be its intersection with $\hat{G}$.
Then $\mu(G) / \mu(\hat{U}) > \sigma - \kappa/2  = 1-\kappa$.
If $\xi=(x,y) \in G$
then
$$
g_{m-1} \ccircs g_0 (\xi) = \big( R_{\pi/2} (x), C_{n-1} \ccircs C_0(y) \big)
$$
and therefore
$D (g_{m-1} \ccircs g_0) (\xi) \cdot \vetor{p_1} = \vetor{q_1}$.
In particular the angle
between $D (g_{m-1} \ccircs g_0) (\xi) \cdot F^u_0$ and $F^{cs}_n$ is zero.\qedhere
\end{enumerate}
\end{proof}

\subsection[Dealing with case IV]{Dealing with Case IV}\label{ss.IV}

As already mentioned, the proof of Lemma~\ref{l.case IV} will be essentially reduced to dimension $4$.
Let us fix some notation.
For $t\in \R$, define the following symplectic linear map on $\R^4 = \{(p_1, p_2, q_1, q_2)\}$:
\begin{equation}\label{e.rotation}
R_t = \begin{pmatrix}
\cos t & -\sin t & 0      & 0       \\
\sin t & \cos t  & 0      & 0       \\
0      & 0       & \cos t & -\sin t \\
0      & 0       & \sin t & \cos t
\end{pmatrix} .
\end{equation}

For $t$ in the circle $\nicefrac{\R}{\pi \Z}$, let us
indicate $\nor{t} = \min_{k\in\Z}|t-k\pi|$.

If $v=(p_1,p_2,q_1,q_2)$ is a vector in $\R^4$ such that $(p_1,p_2) \neq (0,0)$
then we let $\Theta(v)$ be such that $(p_1,p_2) = \pm (r \cos\Theta(v), r \sin\Theta(v))$,
where $r = (p_1^2 + p_2^2)^{1/2}$;
thus $\Theta(v)$ is uniquely defined in $\circulo$.

For $\beta>0$, define cones
$$
\cC_ \beta = \{(p_1, p_2, q_1, q_2) \in \R^4; \; \|(q_1, q_2)\| < \beta \|(p_1, p_2)\|\}.
$$

\begin{lemma} \label{l.sheer}
For every $v \in \cC_1$ there is a symplectic linear map $L_v \colon  \R^4 \to \R^4$ such that:
\begin{enumerate}
\item $L_v$ preserves the plane spanned by $\vetor{q_1}$ and $\vetor{q_2}$.
\item $L_v\big(\vetor{p_1}\big)$ is collinear to $v$.
\item $\Theta(L_v(w)) =  \Theta(w) + \Theta(v)$ for all $v$, $w \in \cC_1$.
\item $\|L_v\| = \|L_v^{-1}\| \le K_L$ for all $v\in\cC_1$, where $K_L>1$ is an absolute constant.
\end{enumerate}
\end{lemma}

\begin{proof}
Let $v = (p_1, p_2, q_1, q_2) \in \cC_1$.
Assume that $p_1^2 + p_2^2 = 1$.
Let $\theta = \Theta (v)$.
Then $R_{-\theta} (v) = (1,0,a,b)$ for certain $a$ and $b$ with $a^2+b^2\le 1$.
(Recall~\eqref{e.rotation}.)
The matrix
$$
M =
\begin{pmatrix}
1 & 0 & 0 & 0 \\
0 & 1 & 0 & 0 \\
a & b & 1 & 0 \\
b & 0 & 0 & 1
\end{pmatrix}
$$
is symplectic and preserves $\Theta$.
Then $L_v = R_\theta \circ M$ has the required properties.
\end{proof}

\smallskip

The following well-known fact about random walks will play a important role in the proof:

\begin{lemma}\label{l.probabilistic}
Let $X_0$, $X_1$, \ldots be independent identically distributed random variables,
with $\mathbb{E} |X_0|< \infty$ and $0< \mathbb{E} X_0^2 < \infty$.
Let $S_n = X_0 + \cdots + X_{n-1}$.
For any fixed $K>0$, the probability that $|S_n| \le K$ for all $n$
is zero.
\end{lemma}

\begin{proof}
Let $a$ and $\sigma$ be respectively the mean and the variance of $X_0$.
Of course, $\sigma>0$.
By the Central Limit Theorem, $Y_n = (S_n - an)/(\sigma \sqrt{n})$ converges in distribution to a
standard normal random variable.
That is, 
$$
\lim_{n \to \infty} \P[\alpha \le Y_n \le \beta] =
\frac{1}{\sqrt{2\pi}} \int_\alpha^\beta e^{-t^2/2} \, dt
\qquad \forall \alpha<\beta.
$$
Fix $K>0$.
If $a=0$ then
${\P[|S_n|\le K]} = \P \big[|Y_n|\le \frac{K}{\sigma\sqrt{n}} \big]$.
If $a\neq 0$ then
${\P[|S_n|\le K]} \le \P \big[|Y_n|\ge \frac{|a|n - K}{\sigma\sqrt{n}} \big]$.
In either case, we have
$\lim_{n \to \infty} \P[|S_n|\le K] = 0$.
In particular, $\P[|S_n|\le K \ \forall n] = 0$.
\end{proof}

\begin{proof}[Proof of Lemma~\ref{l.case IV}]

\medskip\noindent\emph{Step 1. Preparation.}
Let $\eps$, $\kappa$, $\KIV$, $\tau$ be given.
Let $\eps'>0$ be so that
$\eps' < (K_\cA \KIV)^{-2} \eps$ (recall the definition of $K_\cA$ from \S\ref{ss.flex TM})
and
$M(\cC_1) \subset \cC_{\tau^2}$ for all linear $M\colon  \R^4 \to \R^4$ with $\|M - \Id\|< \eps'$.
Let $\delta$ be given by $e^{2K_L^2 \delta} -1 = \eps'$
(where $K_L$ comes from Lemma~\ref{l.sheer}).
Let $\alpha>0$ be given by Lemma~\ref{l.alpha rot}
applied with $\eps'$ and $\kappa/10$ in the place of $\eps$ and $\kappa$,
respectively.

Let $\D$ be the open unit ball in $\R^4$.\footnote{A ``box'' as
in \S\ref{ss.preview} would work equally well.}
Let $\bar \mu$ be Lebesgue measure on $\R^4$ normalized so that $\bar{\mu}(\D)=1$.

Choose a  smooth function $H\colon \R^4 \to \R$
not identically zero that vanishes outside of $\D$, and such that $\|D^2 H\| < \delta$.
Let $h\colon \R^4 \to \R^4$ be the associated time $1$ map, that is, $h = \varphi_H^1$.
Let $\nu$ be the probability measure on the circle $\circulo$ defined by
$$
\nu(A) = \bar\mu \big\{x\in \D;\; \Theta \big( Dh (x) \cdot \vetor{p_1} \big) \in A \big\}
\text{, for each Borel set $A \subset \circulo$}\, .
$$
We assume that $H$ was chosen so that the support of $\nu$
is contained in the interval
$\{t; \; \nor{t} < \alpha/20\}$.

Let $X_0, X_1, \ldots$ be independent circle-valued random variables,
all distributed according to the measure~$\nu$.\footnote{It is interesting, although unimportant, to see that $\mathbb{E}(\tan X_0) = 0$.}
Consider the random walk $S_n = X_0 + \cdots + X_{n-1}$.
By Lemma~\ref{l.probabilistic}, there exists
$m_1$ such that
\begin{equation}\label{e.def m}
\text{the probability that
$\nor{S_n - \tfrac{\pi}{2}} > \tfrac{\alpha}{20}$  for all  $n \le m_1$
is less than $\tfrac{\kappa}{20}$.}
\end{equation}
We will show that $m_1$ has the desired properties.

Take $m \ge m_1$ and assume that $\{z, \ldots, f^m z\}$
is a segment of type~IV with constants $\KIV$, $\tau$.
Let $\cL_i \colon  T_{f^i z} M \to \R^{2N}$
and $A_i  = \cL_{i+1} \circ Df(f^i z) \circ \cL_i$
be as in the definition of type IV.
We want to prove that the split sequence
$Df(f^i z)\colon {E^u \oplus E^{cs} \hookleftarrow}$ ($0 \le i < m$)
is $(\eps, \kappa)$-flexible.
Bearing in mind Lemma~\ref{l.flex change coord},
it suffices to show that the split sequence on~$\R^{2N}$
\begin{equation}\label{e.split seq}
\left\{ F_i^u \oplus F_i^{cs} \xrightarrow{A_i} F_{i+1}^u \oplus F_{i+1}^{cs} \right\}_{0\le i \le m}
\quad \text{where}  \quad
F_i^* = \cL_i(E^*(f^i z))
\end{equation}
is $(\eps', \kappa)$-flexible.

By definition of type~IV,
$$
A_i (x,y) = \big(B_i(x), C_i(y)\big), \quad
B_i(p_1,p_2,q_1,q_2) = (c_i p_1, c_i p_2, c_i^{-1} q_1, c_i^{-1} q_2), \quad
c_i >\tau \, .
$$
Also, for all $i$,
\begin{equation}\label{e.4 vectors}
\vetor{p_1} \in F_i^u, \quad
\vetor{p_2}, \  \vetor{q_1}, \ \vetor{q_2} \in F_i^{cs}
\end{equation}

\medskip\noindent\emph{Step 2. Reduction to $\R^4$.}
Let $U_0 = \D$ and
$U_n = B_{n-1} \ccircs  B_0 (\D)$ for $1 \le n \le m$.

\begin{subl}\label{sl.reduction}
There exist symplectomorphisms $g_n \colon  \R^4 \to \R^4$, for $0 \le n < m$ with the following properties:
\begin{enumerate}
\item $g_n$ equals $B_n$ outside $U_n$ and $\|D(B_n^{-1} \circ g_n) - \Id \| < \eps'$ at each point;
\item for each $n$, there is a smooth function $H_n\colon  \R^4 \to \R$ constant outside $U_n$ such that
$\|D^2 H\| < K_L^2 \delta$ and the time $1$ map $\varphi^1_{H_n}$ equals $B_n^{-1} \circ g_n$;
\item there is a set $G \subset U_0$
with (normalized) measure $\bar \mu(G) > 1- \kappa/10$
such that
\begin{equation}\label{e.90 deg}
\bignor{ \Theta \big( D (g_{m-1} \ccircs g_0)(x) \cdot \vetor{p_1} \big) - \tfrac{\pi}{2} } < \tfrac{\alpha}{2}
\quad \text{for all $x \in G$.}
\end{equation}
\end{enumerate}
\end{subl}

Let us assume the sublemma for a while and see how to
conclude the proof of Lemma~\ref{l.case IV}.
Let $\gamma>0$ be given (as in the definition of flexibility).
We will assume $2N > 4$, leaving for the reader the easy adaptation for the $4$-dimensional case.
Consider the hamiltonians $H_n$ given by Sublemma~\ref{sl.reduction}, and
apply Lemma~\ref{l.dim reduction} with $2\nu=4$, $K_L^2 \delta$ in the place of $\delta$,
and $\kappa/10$ in the place of $\kappa$.
We obtain a cylinder $\hat{U} \subset \R^{2N}$ and
hamiltonians $\hat{H}_n \colon  \R^{2N} \to \R$ such that
writing $\hat{g}_n = A_n \circ \varphi_{\hat{H}_n}^1$ we have:
\begin{itemize}
\item $\hat{g}_n$ equals $A_n$ outside of $A_{n-1} \ccircs A_0( \hat{U} )$;
\item $\|D^2 \hat{H}_n\| < 2K_L^2 \delta$ and hence, by Lemma~\ref{l.ode}, $\|D(A_n^{-1} \circ \hat{g}_n) - \Id\| < \eps'$;
\item there is a set $\hat{G} \subset \hat{U}$ with $\mu(\hat{G}) > (1-\kappa/10) \mu(\hat{U})$
such that if $\xi = (x,y) \in \hat{G}$ then
$$
\hat{g}_{m-1}  \ccircs \hat{g}_0 (\xi) =\\
\big(g_{m-1} \ccircs g_0 (x) , \,
C_{m-1} \ccircs C_0(y) \big) \, .
$$
\end{itemize}
Since $\hat{U}$ is a cylinder,
the set $\{(x,y) \in \hat{U}; \; x \in G\}$ has measure $> (1-\kappa/10) \mu(\hat{U})$;
let $G_1$ be its intersection with $\hat{G}$.
Then $\mu(G_1) > (1-2\kappa/10)\mu(\hat{U})$.
If $\xi = (x,y) \in G_1$ then by \eqref{e.90 deg},
the angle between the vector $D(g_{m-1}  \ccircs g_0) (x) \cdot \vetor{p_1}$ in $\R^4$
and the space spanned by $\vetor{p_2}$, $\vetor{q_1}$, $\vetor{q_2}$ is at most $\alpha$.
Using~\eqref{e.4 vectors}
we conclude that
$$
\angle \left(D(\hat{g}_{m-1}  \ccircs \hat{g}_0) (\xi) \cdot F^u_0, F^{cs}_m \right) < \tfrac{\alpha}{2}
\quad \text{for all $\xi \in G_1$.}
$$
We need to perform a last perturbation $\hat{g}_m$ to make the angle smaller than~$\gamma$.

\smallskip

Let $\gamma'$ be very small.
By Vitali's Lemma, we can find a finite family of
disjoint small euclidian balls $D_\ell$ contained in the open set $G_1$
and whose union leaves out a set of measure at most $(1-\kappa/10)\mu(\hat{U})$.
In fact, the balls are taken small enough so that
the variation of the angle
$\angle \left(D(\hat{g}_{m-1}  \ccircs \hat{g}_0) (\xi) \cdot F^u_0, F^{cs}_m \right)$
is less than $\gamma'$ when $\xi$ runs over $D_\ell$.
For each $\ell$, let $\xi_\ell$ be the center of the ball $D_\ell$,
and let $v_\ell$ be the vector $D(\hat{g}_{m-1}  \ccircs \hat{g}_0) (\xi_\ell) \cdot \vetor{p_1}$.

We now use the definition of $\alpha$.
For each $\ell$, Lemma~\ref{l.alpha rot} applied to the set $D'_\ell = \hat{g}_{m-1} \ccircs \hat{g}_0 (D_\ell)$
gives a symplectomorphism $h_\ell\colon  \R^{2N}\to\R^{2N}$ with the following properties:
\begin{itemize}
\item $h_\ell$ equals the identity outside of $D'_\ell$;
\item $\|Dh_\ell - \Id \|< \eps'$;
\item there is a set $G'_\ell \subset D'_\ell$  with $\mu(G'_\ell) > (1-\kappa/10) \mu(D'_\ell)$ such that
for every $\xi' \in G'_\ell$, the vector $Dh_\ell(\xi') \cdot v_\ell$ belongs to $F^{cs}_m$.
\end{itemize}

Let $G_0 = \bigsqcup_\ell (\hat{g}_{m-1} \ccircs \hat{g}_0)^{-1} (G'_\ell)$.
Then $\mu(G_0) > (1-\kappa)\mu(\hat{U})$.
Finally, define the perturbation $\hat{g}_m$
as equal to $A_m \circ h_\ell$ in each $D'_\ell$, and equal
to $A_m$ outside.
If $\gamma'$ was chosen sufficiently small then
for every $\xi \in G_0$ we have
$$
\angle \left(D(\hat{g}_m \ccircs \hat{g}_0) (\xi) \cdot F^u_0, F^{cs}_{m+1} \right) < \gamma.
$$
This shows that the split sequence~\eqref{e.split seq} is $(\eps', \kappa)$-flexible.
Hence to complete the proof of Lemma~\ref{l.case IV} we are left to prove
Sublemma~\ref{sl.reduction}.

\medskip\noindent\emph{Step 3. Definition of perturbations in $\R^4$.}
Before starting the proof of the sublemma,
notice the first condition there implies that
\begin{equation}\label{e.inv cone}
Dg_n(x) (\cC_1) \subset \cC_1 \quad \forall x \, ,
\end{equation}
due to the definition of $\eps'$ and the fact that $B_n (\cC_{\tau^2} ) \subset \cC_1$.

Let $\cN(v)$ indicate $v / \|v\|$.
Fix a constant $K>1$
such that for all unit vectors $v$, $w \in \cC_1$ we have:
\begin{equation} \label{e.const chata}
\begin{aligned}
\nor{\Theta(v)-\Theta(w)} &\le K \|v-w\| \\
\big \|\cN(Dg_n(x) \cdot v)  - \cN(Dg_n(x) \cdot w) \big\| &\le K \|v-w\| \quad
\forall x\,
\end{aligned}
\end{equation}
(provided $g_n$ complies with the first condition in Sublemma~\ref{sl.reduction}).
Let
\begin{equation}\label{e.def eta}
\eta = \min \big( \tfrac{\alpha}{100 K^2 m}, \tfrac{\kappa}{20 m} \big) \, .
\end{equation}

\smallskip

For each $n = 0$, \ldots, $m$,
we are also going to define a finite family $\{D_i\}_{i \in I_n}$ of disjoint subsets of $U_n$.
Also, the sets of indices $I_0$, \ldots, $I_m$ will be disjoint,
and each $I_n$ will be partitioned as
$I_n = I_n^\text{arrived} \sqcup I_n^\text{not yet}$.

Start defining $g_0 = B_0 \circ h$
(recall the definition of $h$ in step~1).
Then, by Lemma~\ref{l.ode},
$\|D(B_0^{-1} \circ g_0) - \Id \| < e^\delta - 1 < \eps' $, as required.
Also define $I_0 = I_0^\text{not yet} = \{0\}$, $D_0 =\D$.

By induction, assume that $g_0$, \ldots, $g_{n-1}$ and $\{D_i\}_{i \in I_{n-1}}$
are already defined, for some $n$ with $0 < n \le m$,
and let us proceed to define $g_n$ (if $n < m$) and $\{D_i\}_{i \in I_n}$.
First define a vector field $\v_n$ on $\R^4$ by
$$
\v_n ( g_{n-1} \ccircs g_0 (x)) = \cN \left( D(g_{n-1} \ccircs g_0)(x) \cdot \vetor{p_1} \right) \, .
$$
Then $\v_n$ takes values on the cone $\cC_1$, because~\eqref{e.inv cone} holds for $g_0$, \ldots, $g_{n-1}$.

Let $V_{n-1} = \bigsqcup_{i \in I_{n-1}} D_i \subset U_{n-1}$,
so that $g_{n-1}(V_{n-1}) \subset U_n$.
For $x \in g_{n-1}(V_{n-1})$ and $r>0$, define a neighborhood of $x$ by
$$
\tilde D(x, r, n) = \{x + r L_{\v_n(x)} (y); \; y \in \D\}
$$
(where the $L$'s come from Lemma~\ref{l.sheer}).
These neighborhoods are ``quasi-round'', in the sense that
$B_{K_L^{-1} r}(x) \subset \tilde D(x,r,n) \subset B_{K_L r}(x)$.
Now consider the family of sets $\tilde D(x,r,n)$
with $r$ sufficiently small so that the variation of $\v_n$ in each $\tilde D(x,r,n)$
is less than $\eta$.
This family constitutes a Vitali cover
of the set $g_{n-1}(V_{n-1})$.
Therefore we can find a finite subfamily
$\{D_i = \tilde D (\xi_i, r_i, n)\}_{i \in I_n}$
whose disjoint union covers most of the set, that is,
\begin{equation}\label{e.eta loss}
\bar\mu \big(g_{n-1}(V_{n-1}) \setminus V_n \big) < \eta, \quad \text{where} \quad
V_n = {\textstyle \bigsqcup_{i\in I_n} D_i} \, .
\end{equation}

So we have defined the set of indices $I_n$ and the family of sets $\{D_i\}_{i \in I_n}$.
Let $I_n^\text{arrived}$ be the set of $i \in I_n$
such that at least one of the following two properties is satisfied:
\begin{itemize}
\item $\nor{\Theta(\v_n(\xi_i)) - \tfrac{\pi}{2}} < \tfrac{\alpha}{10}$;
\item if $i'$ denotes the unique index in $I_{n-1}$ such that $D_i \subset g_{n-1}(D_{i'})$ 
then $i'$ already belongs to $I_{n-1}^\text{arrived}$.
\end{itemize}
Let $I_n^\text{not yet} = I_n \setminus I_n^\text{arrived}$.

Next we define $g_n$ (in the case $n<m$).
Let $g_n$ be equal to $B_n$ outside of $\bigsqcup_{i \in I_n^\text{not yet}} D_i$.
Inside each domain $D_i$ with $i\in I_n^\text{not yet}$, let
$g_n = {B_n \circ T_i^{-1} \circ h \circ T_i}$,
where
$$
T_i \colon  D_i \to \D \quad \text{is given by} \quad
T_i(x) =  L_{\v_n(\xi_i)}^{-1} \big((x-\xi_i) / r_i \big) \, .
$$
Since $T_i$ is an affine map that expands the symplectic form by a constant factor,
$g_n$ is a well-defined symplectomorphism of $\R^4$.

Let us see that $g_n$ satisfies parts (1) and (2) from Sublemma~\ref{sl.reduction}.
Let
\begin{equation} \label{e.random hamiltonians}
H_n (x) =
\begin{cases}
r_i^{-2} H(T_i(x)) &\text{if $x\in D_i$ with $i \in I_n^\text{not yet}$,}\\
0 &\text{otherwise.}
\end{cases}
\end{equation}
It follows from Lemma~\ref{l.hamilton change coord} that
the time $1$ map $\varphi_{H_n}^1$ is precisely $B_n^{-1} \circ g_n$.
The lemma also gives that $\|D^2 H_n \| \le K_L^2 \|D^2 H\| < K_L^2 \delta$.
This shows part (2) of Sublemma~\ref{sl.reduction}.
Recalling Lemma~\ref{l.ode}, one sees that the first part follows from the second.

\smallskip

To summarize, we have defined the maps $g_n$ (together with other objects)
and have verified that they satisfy properties~(1) and (2) of Sublemma~\ref{sl.reduction}.
Next we will show that property~(3) also holds.

\medskip\noindent\emph{Step 4. Random walk behavior.}
Recall that we have defined in step~1 circle-valued random variables $X_n$.
We will only be interested in the first $m$ of them.
Let us choose a probability space for these variables
(as well as their sums $S_n = X_0 + \cdots + X_{n-1}$)
to ``live in'':
it is $(\Omega, \P)$, where $\Omega = \D^m$ and $\P = \bar\mu^m$.
Let now each random variable $X_n$ be the function
$$
X_n\colon  \Omega \to \circulo \quad \text{given by} \quad
X_n(\omega_0,\ldots,\omega_{m-1}) = \Theta\big(Dh(\omega_n) \cdot \vetor{p_1}\big).
$$

In imprecise words, we will see that the angles $\Theta(\v_n(\cdot))$ behave approximately
like the random walk $S_n$, with an absorbing barrier around $\pi/2$.
This and \eqref{e.def m} will permit us to
show the third part of Sublemma~\ref{sl.reduction}.

\smallskip

In what follows,
let $\mathbf{L}(c)$ stand for an unspecified $t \in \circulo$ with $\nor{t} < c$.
By construction, if $x$ and $x'$ both belong to the same $D_i$ with $i\in I_n$ then
$\| \v_n(x_n) - \v_n(x_n') \| < \eta$
and so \eqref{e.const chata} implies
$\Theta(\v_n(x)) = \Theta(\v_n(x')) + \mathbf{L}(K\eta)$.

\smallskip

An \emph{itinerary} is a sequence $\ivec = (i_0, i_1, \ldots, i_m) \in I_0 \times \cdots \times I_m$
such that $D_{i_{n+1}} \subset g_n(D_{i_n})$ for $0 \le n < m$.
(In fact, $\ivec$ is uniquely determined by $i_m$.)
A \emph{pseudo-orbit with itinerary $\ivec = (i_n)$}
is a sequence $(x_1, \ldots, x_m)$ such that
$x_n \in D_{i_n}$ for each $n$.
One example is the orbit $(x_n) = (g_{n-1} \ccircs g_0(x_0))$
of a point $x_0$ in $(g_{m-1}\ccircs g_0)^{-1}(D_{i_m})$.
Other example of pseudo-orbit is $(\xi_{i_1}, \ldots, \xi_{i_m})$.
(Recall $\xi_i$ is the ``center'' of $D_i$).

All pseudo-orbits with itinerary $\ivec = (i_n)$ are of the form
\begin{equation}\label{e.dict}
(x_1, \ldots, x_m) =
\big( g_0(\omega_0), g_1(T_{i_1}^{-1}(\omega_1)), \ldots, g_{m-1}( T_{i_{m-1}}^{-1} (\omega_{m-1})) \big)
\end{equation}
for some $\omega = (\omega_n) \in \D^m = \Omega$.
With this writing, we claim that
\begin{align}
\Theta \big( \v_1(x_1) \big) &= X_0(\omega)  \label{e.recurrence 1} \\
\Theta \big(\v_{n+1}(x_{n+1}) \big) &=
\begin{cases}
\Theta(\v_n(x_n)) + \mathbf{L}(K \eta)                  &\text{if $i_n \in I_n^\text{arrived}$,}\\
\Theta(\v_n(x_n)) + X_n(\omega) + \mathbf{L}(2 K^2\eta) &\text{if $i_n \in I_n^\text{not yet}$.}
\end{cases} \label{e.recurrence n}
\end{align}

The proof of~\eqref{e.recurrence 1} is immediate:
$$
\Theta \big( \v_1(x_1) \big) =
\Theta \big( Dg_0 ( g_0^{-1}(x_1)) \cdot \vetor{p_1} \big) =
\Theta \big( Dg_0 ( \omega_0 ) \cdot \vetor{p_1} \big) =
X_0(\omega).
$$
Now take $n$ with $1 \le n \le m -1$.
We have
$$
\v_{n+1}(x_{n+1}) =
\cN \big( Dg_n (g_n^{-1}(x_{n+1})) \cdot \v_n(g_n^{-1}(x_{n+1})) \big) \, .
$$
Notice that the point $g_n^{-1}(x_{n+1})$  
belongs to $D_{i_n}$.
If $i_n \in I_n^\text{arrived}$ then $g_n$ restricted to $D_{i_n}$
equals $B_n$, which preserves $\Theta$, therefore
$$
\Theta \big( \v_{n+1}(x_{n+1}) \big)
= \Theta \big(\v_n(g_n^{-1}(x_{n+1})) \big)
= \Theta \big(\v_n(x_n) \big) + \mathbf{L}(K \eta),
$$
proving the first part of~\eqref{e.recurrence n}.
For $i_n \in I_n^\text{not yet}$
we have
$$
Dg_n(g_n^{-1}(x_{n+1})) =
B_n \circ L_{\v_n(\xi_{i_n})} \circ Dh(\omega_n) \circ L_{\v_n(\xi_{i_n})}^{-1} \, .
$$
Lemma~\ref{l.sheer}
leads therefore to
$$
\Theta \big(Dg_n (g_n^{-1}(x_{n+1})) \cdot \v_n(\xi_{i_n})\big)
= \Theta \big(Dh(\omega_n) \cdot \vetor{p_1}\big) + \Theta (\v_n(\xi_{i_n}))
= X_n (\omega) + \Theta (\v_n(\xi_{i_n}))
\, .
$$
Therefore, using that the points $g_n^{-1}(x_{n+1})$, $\xi_{i_n}$, and $x_n$ belong to the same $D_{i_n}$,
we can write:
\begin{align*}
\Theta \big( \v_{n+1}(x_{n+1}) \big)
&= \Theta \big(Dg_n (g_n^{-1}(x_{n+1})) \cdot \v_n(g_n^{-1}(x_{n+1})) \big)   \\
&= \Theta \big(Dg_n (g_n^{-1}(x_{n+1})) \cdot \v_n(\xi_{i_n})\big) + \mathbf{L}(K^2\eta) \\
&= X_n(\omega) + \Theta (\v_n(\xi_{i_n})) + \mathbf{L}(K^2 \eta) \\
&= X_n(\omega) + \Theta (\v_n(x_n)) + \mathbf{L}(2K^2 \eta) \, ,
\end{align*}
This completes the proof of the claim~\eqref{e.recurrence n}.

\smallskip

Still assuming $(x_n)$ and $(\omega_n)$ as in \eqref{e.dict},
we now claim that:
\begin{alignat}{2}
\text{if} \quad i_m \in I_m^\text{arrived} \quad
&\text{then} &\quad
&\Theta(\v_m(x_m)) = \tfrac{\pi}{2} + \mathbf{L}(\tfrac{\alpha}{2}) \label{e.arrive}\\
&\text{else} &\quad
&\nor{S_n(\omega)- \tfrac{\pi}{2}} > \tfrac{\alpha}{20} \ \text{for all $n$.}\label{e.never arrive}
\end{alignat}
If $i_m \in I_m^\text{arrived}$ then let $n_0$ be the least such that
$i_{n_0} \in I_{n_0}^\text{arrived}$.
It follows from the definitions that
$$
\Theta(\v_{n_0}(\xi_{i_{n_0}})) = \tfrac{\pi}{2} + \mathbf{L}(\tfrac{\alpha}{10}) \quad \text{and} \quad
i_n \in I_n^\text{arrived} \text{ for all $n\ge n_0$.}
$$
Using repeatedly \eqref{e.recurrence n}, together with~\eqref{e.def eta},
the claim~\eqref{e.arrive} follows.
On the other hand, if $i_m \in I_m^\text{not yet}$ then $i_n \in I_n^\text{not yet}$ for all $n$.
Using \eqref{e.recurrence 1} and \eqref{e.recurrence n} for the pseudo-orbit $(\xi_n)$,
and also~\eqref{e.def eta}, we obtain
$$
\Theta(\v_n(\xi_{i_n})) = S_n(\omega) + \mathbf{L}(\tfrac{\alpha}{50}) \, .
$$
The fact that  $i_n \in I_n^\text{not yet}$  also implies that
$\nor{\Theta(\v_n(\xi_{i_n})) - \tfrac{\pi}{2}} \ge \tfrac{\alpha}{10}$,
so~\eqref{e.never arrive} follows.

\smallskip

Next, for each itinerary $\ivec = (i_n)$,
define the following subset of $\Omega$:
$$
W_\ivec =
g_0^{-1}(D_{i_1}) \times T_{i_1}(g_1^{-1}(D_{i_2})) \times  \cdots \times
T_{i_{m-1}}(g_{m-1}^{-1}(D_{i_m})) .
$$
Let us evaluate its probability.
Using that $g_n$'s preserve $\bar\mu$ and that
the affine maps $T_i \colon  D_i \to \D$ expand $\bar\mu$ by the factor $\det T_i = 1/\bar\mu(D_i)$,
we get:
$$
\P (W_\ivec) = \bar\mu(D_{i_1}) \det(T_{i_1}) \bar\mu(D_{i_2}) \cdots \det(T_{i_{m-1}}) \bar\mu(D_{i_m}) = \bar\mu(D_{i_m}).
$$
Summing over the itineraries such that $i_m \in I_m^\text{not yet}$,
using~\eqref{e.never arrive} and \eqref{e.def m},
we obtain:
\begin{align}
{\textstyle \sum_{i_m \in I_m^\text{not yet}} } \bar\mu(D_{i_m})
&=   \P \big( {\textstyle \bigsqcup_{i_m \in I_m^\text{not yet}} W_\ivec } \big) \notag \\
&\le \P \left[\nor{S_n - \tfrac{\pi}{2}} > \tfrac{\alpha}{20} \text{ for all } n \le m \right] \notag \\
&<   \tfrac{\kappa}{20} \, . \label{e.punch}
\end{align}

Consider the union of all $D_{i_m}$ with $i_m \in I_m$,
that is, the set $V_m$.
It follows from~\eqref{e.eta loss} and \eqref{e.def eta} that
$$
\bar{\mu}(V_m) > 1-m\eta > 1-\tfrac{\kappa}{20} \, .
$$
Hence \eqref{e.punch} implies that
the union $G'$ of all $D_{i_m}$ with $i_m \in I_m^\text{arrived}$
has measure $\bar\mu(G') > 1-\tfrac{\kappa}{10}$.
Let
$$
G = (g_{m-1} \cdots g_0)^{-1}(G'), \quad \text{so that }
\bar\mu(G) = \bar\mu(G') > 1-\tfrac{\kappa}{10}.
$$
If $x \in G$
then \eqref{e.arrive} applied to the orbit $(x_n) = (g_{n-1} \ccircs g_0)(x)$
gives
$\nor{\Theta ( \v_m(x_m)) - \tfrac{\pi}{2}} < \tfrac{\alpha}{2}$,
which is precisely~\eqref{e.90 deg}.
This proves part~(3) of Sublemma~\ref{sl.reduction} and hence Lemma~\ref{l.case IV}
(and the Main Lemma).
\end{proof}

\section[Exploiting flexibility]{Exploiting Flexibility} \label{s.flex III}

With the Main Lemma,
Theorem~\ref{t.continuity} is proven
following \cite{BV Annals}.
For the first part of the proof,
we explain in \S\ref{ss.lowering} how the arguments from \cite{BV Annals} can be adapted.
The second part could be done repeating parts of \cite{BV Annals} almost word for word.
However, we present (\S\ref{ss.globalization}) a new and significantly simpler proof,
following suggestions by A.~Avila.

\smallskip

Given $f \in \Diff_\omega^1(M)$, $p \in \{1,\ldots,N\}$ and $m \in \N$,
let $\Gamma_p (f, m)$ be the (open) set of points $x$ such that there is no
$m$-dominated splitting of index $p$ along the orbit of $x$.

The symplectomorphism $f$ is called \emph{aperiodic}
if the measure of the set of its periodic points is zero.
By Robinson's~\cite{Robinson} symplectic version of the Kupka--Smale Theorem,
the generic $f$ has countably many periodic points and in particular is aperiodic.

\smallskip

Let
$$
\Lambda_p(f,x) = \sum_{i=i}^p \lambda_i(f,x)
= \lim_{n \to \infty} \frac{1}{n} \log \|\mathord{\wedge}^p (Df^n(x)) \| \, .
$$
(The reader should recall relations between exterior products
and Lyapunov exponents, see e.g.~\cite[{\S}2.1.2]{BV Annals}.)

\smallskip

\subsection[Lowering the norm along an orbit segment]{Lowering the Norm along an Orbit Segment} \label{ss.lowering}

As consequence of the Main Lemma, we can perturb the map $f$ on
a neighborhood of an orbit segment of length $n$ in such a way that $\|\wedge^p Df^n\|$ drops.
In precise terms:

\begin{lemma}
\label{l.lower}
Let $f \in \Diff_\omega^1(M)$ be aperiodic, $\cV$ be a neighborhood of $f$,
$\delta>0$, and $0< \kappa <1$.
If $m \in \N$ is sufficiently large,
then there exists a measurable function $N\colon  \Gamma_p(f,m) \to \N$ with the following properties.

For a.~e.~$x \in \Gamma_p(f,m)$ and every $n \geq N(x)$,
there exists $r = r(x,n)>0$
such that the following holds:
First, the iterates
$f^j(\bar{B}_r(x))$, for $0 \le j \le n$, are pairwise disjoint.
Second, for any $ 0<r'<r$
there exists $g \in \cV$
such that:
\begin{enumerate}
\item $g$ equals $f$ outside $\bigsqcup_{j=0}^{n-1} f^j(B_{r'}(x))$;
\item there is a set $G \subset B_{r'}(x)$ such that
$\mu(G) > (1-\kappa)\mu(B_{r'}(x))$ and
$$
\frac{1}{n} \log \| \mathord{\wedge}^p (Dg^n(y) ) \|  \, \le \,
\frac{\Lambda_{p-1}(f,x) + \Lambda_{p+1}(f,x)}{2} + \delta \quad
\text{for all $y \in G$.}
$$
\end{enumerate}
\end{lemma}

We remark that the lemma corresponds to \cite[Proposition~4.2]{BV Annals},
giving at the same stroke the conclusions of \cite[Lemma~4.13]{BV Annals}.

\begin{proof}
Denote
\begin{equation}\label{e.Phi}
\Phi(x)   
= \frac{\Lambda_{p-1}(f,x) + \Lambda_{p+1}(f,x)}{2} \, .
\end{equation}
Let $\eps = \eps(f,\cV)$ be given by Lemma~\ref{l.realization}.
Let $m \in \N$ be sufficiently large so that the conclusion of the Main Lemma holds
(with $\kappa/2$ in the place of~$\kappa$).

For the points $x \in \Gamma_p(f,m)$ that are non periodic, Oseledets regular, and
have $\lambda_p(f,x) = \lambda_{p+1}(f,x)$,
the conclusion of the lemma is trivial:
first take $N(x)$ large so that if $n \ge N(x)$ then
$\frac{1}{n} \log \| \mathord{\wedge}^p (Df^n(x) ) \|$ is $\delta/2$-close to $\Lambda_p(f,x) = \Phi(x)$.
Then for each $n \ge N(x)$, take $r=r(x,n)$ small so that
the ball $\bar{B}_r(x)$ is disjoint from its $n$ first iterates and
$Df^n(y)$ is close to $Df^n(x)$ for all  $y \in B_r(x)$.
Letting $g=f$, all the desired conclusions of the lemma hold.

Next consider the set $\Gamma$ formed by the points $x \in \Gamma_p(f,m)$
that are non-periodic, Oseledets regular, and such that $\lambda_p(f,x) > \lambda_{p+1}(f,x)$.
That is, $\Gamma$ is the intersection of $\Gamma_p (f, m)$
with the set $\Sigma_p (f)$ introduced in \S\ref{ss.special seq}.
Assume that $\mu(\Gamma)>0$, otherwise there is nothing left to prove.
Let $A \subset \Sigma_p(f)$ be the set of points such that the non-domination
condition~\eqref{e.nondominated} holds.
Then
$\Gamma = \bigcup_{n \in \Z} f^n (A)$
(because the splitting $E^u \oplus E^{cs}$ over the set $\Sigma_p(f) \setminus \bigcup_{n \in \Z} f^n (A)$
is $m$-dominated of index~$p$).
Fix $C > \sup_{g \in\cV, \ x \in M} \| Dg(x) ^{\pm 1}\|$.

\begin{subl}\label{sl.dirty work}
There exists a measurable function $N\colon  \Gamma \to \N$
such that for a.~e.~$x \in \Gamma$ and for every $n \ge N(x)$,
there exists $\ell$ with $0 < \ell < n-m$ such that $z = f^\ell x$ belongs to $A$
and the following holds:
If $L_i \colon  T_{f^i z} M \to T_{f^{i+1} z}M$, where $0\le i \le m-1$,
are linear maps such that $\|L_i^{\pm 1} \| \le C$
and
\begin{equation}\label{e.collide}
L_{m-1} \cdots L_0 \cdot E^u(z) \cap E^{cs}(f^m z) \neq \{0\}
\end{equation}
then
$$
\frac{1}{n}
\log \left\| \mathord{\wedge}^p \big[
Df^{n- \ell - m}(f^{\ell+m} x) \, L_{m-1} \cdots L_0 \, Df^\ell(x) \big] \right\|
< \Phi(x) + \frac{\delta}{2} \, .
$$
\end{subl}

\begin{proof}
It is contained in the proof of \cite[Proposition~4.2]{BV Annals}.
\end{proof}

Let $x\in \Gamma$ be fixed from now on, and let $n \ge N(x)$, $\ell = \ell(x,n)$, and $z=f^\ell x$
be as in Sublemma~\ref{sl.dirty work}.
By mere continuity, we can weaken the requirement~\eqref{e.collide}
to a small angle condition.
More precisely, there exists $\gamma = \gamma(x,n) > 0$ with the following properties:
Given points $y_0, \ldots, y_n \in M$ such that
$$
d(y_i , f^i x) < \gamma \ \forall i \quad \text{and} \quad
f(y_i) = y_{i+1}  \ \forall i \in \{0,\ldots,n-1\} \setminus \{\ell,\ldots,\ell+m-1\},
$$
and given linear maps $\tilde L_i \colon  T_{y_{\ell+i}} M \to T_{y_{\ell+i+1}}M$, for $0 \le i \le m-1$,
such that $\|\tilde L_i^{\pm 1}\| \le C$ and
\begin{equation}\label{e.small angle hyp}
\angle \big( \tilde L_{m-1} \cdots \tilde L_0 \cdot \mathrm{i}_{y_\ell}^{z} \cdot E^u(z) \, ,
\mathrm{i}_{y_{\ell+m}}^{f^m z} \cdot E^{cs}(f^m z) \big) < \gamma \, ,
\end{equation}
(recall \eqref{e.i})
then
\begin{equation}\label{e.small angle conseq}
\frac{1}{n}
\log \left\| \mathord{\wedge}^p \big[
Df^{n- \ell - m}(y_{\ell+m}) \, \tilde L_{m-1} \cdots \tilde L_0  \, Df^\ell(y_0) \big] \right\|
< \Phi(x) + \delta \, .
\end{equation}

Since $z$ belongs to $A$, the Main Lemma says that
the split sequence
$Df(f^i z) \colon {E^u \oplus E^{cs} \hookleftarrow}$ ($0\le i < m$)
is $(\eps, \kappa)$-flexible.
Let $r_0$ be the radius $r(z,\gamma)$
given by Lemma~\ref{l.realization}.
Since $z$ is not periodic, there is $r>0$ be such that for $0 \le j \le n$,
$f^j(\bar{B}_r(x))$ is contained in $B_\gamma(f^i x)$
and does not intersect $\bar{B}_r(x)$.
Let us see that $r$ has the required properties.

Given $r'$ with $0<r'<r$,
let $U = f^\ell( B_{r'}(x))$.
By Lemma~\ref{l.realization},
there exist $g \in \cV$ and $\hat G \subset U$
such that:
\begin{enumerate}
\item $g$ equals $f$ outside $\bigsqcup_{j=0}^{n-1} f^j(U)$;
\item $\mu(\hat G) > (1-\kappa)\mu(U)$;
\item $\angle \left( Dg^n(\xi) \cdot \mathrm{i}_\xi^z \cdot E^u(z), \,
              \mathrm{i}_{g^n \xi}^{f^n z} \cdot E^{cs}(f^m z) \right)< \gamma$ for every $\xi \in \hat{G}$.
\end{enumerate}

Now let $G = f^{-\ell}(\hat{G}) \subset B_{r'}(x)$.
For any $y \in G$, if we define
$y_i = g^i y$ for $0 \le i \le n$,
and $\tilde L_i = Dg(y_{\ell+i})$ for $0 \le i < m$
then relation~\eqref{e.small angle hyp} holds.
Therefore so does~\eqref{e.small angle conseq}, that is,
$\frac{1}{n} \log \| \mathord{\wedge}^p (Dg^n(y) ) \|
\le \Phi(x) + \delta$,
as we wanted to show.
\end{proof}

\subsection{Globalization}\label{ss.globalization}

The next step in the proof is to construct a global perturbation of $f$
that exhibits a drop in some integrated exponent
$\LE_p(f) = \int \Lambda_p(f)$.
Let $\Gamma_p (f, \infty)$ be the set of points $x$ such that there is no dominated splitting of index $p$
along the orbit of $x$; that is,
$\Gamma_p (f, \infty) = \bigcap_{m\in \N} \Gamma_p(f, m)$.

\begin{prop}
\label{p.jump}
Given an aperiodic diffeomorphism $f \in \Diff_\omega^1(M)$, let
$$
J_p(f) = \int_{\Gamma_p(f,\infty)} \frac{\lambda_p(f,x) - \lambda_{p+1}(f,x)}{2} \ d\mu(x).
$$
Then for any neighborhood $\cV$ of $f$ and any $\delta>0$,
there exists $g\in\cV$ such that
\begin{equation}\label{e.jump}
\LE_p(g) < \LE_p(f) - J_p(f) + \delta.
\end{equation}
\end{prop}

\begin{proof}
Let $f$ and $\delta$ be given.
Let $\Phi$ be given by \eqref{e.Phi}.
We are going to show that there exists $m \in \N$ and $g$ arbitrarily $C^1$-close to $f$
that equals $f$ outside the open set $\Gamma_p(f,m)$ and such that
\begin{equation}\label{e.global}
\int_{\Gamma_p(f,m)} \Lambda_p(g) <
\delta +
\int_{\Gamma_p(f,m)} \Phi.
\end{equation}
Let us postpone the proof and see how \eqref{e.global} implies the proposition.
We have:
\begin{alignat*}{2}
\int \Lambda_p(g)
&=   \int_{\Gamma_p(f,m)} \Lambda_p(g) +  \int_{M \setminus \Gamma_p(f,m)} \Lambda_p(f)
&\quad&\text{(because $g=f$ outside $\Gamma_p(f,m)$)} \\
&\le \delta + \int_{\Gamma_p(f,m)} \Phi +  \int_{M \setminus \Gamma_p(f,m)} \Lambda_p(f)
&\quad&\text{(by \eqref{e.global})} \\
&\le \delta + \int_{\Gamma_p(f,\infty)} \Phi +  \int_{M \setminus \Gamma_p(f,\infty)} \Lambda_p(f)
&\quad&\text{(since $\Gamma_p(f,m) \supset \Gamma_p(f, \infty)$ and $\Phi \le \Lambda_p(f)$)} \\
&=   \delta - J_p(f) + \int_{M} \Lambda_p(f)  \, ,
\end{alignat*}
which is \eqref{e.jump}.

Let us see how to construct $g$.
Let $\kappa=\delta$.
Take $m \in \N$ large enough so that Lemma~\ref{l.lower} applies and gives a function
$N\colon \Gamma_p(f,m) \to \N$.
For simplicity, write $\Gamma = \Gamma_p(f,m)$.

\begin{subl}
There is a measurable set $B \subset \Gamma$ such that:
\begin{itemize}
\item The orbit of almost every point in $\Gamma$ visits $B$.
\item for each $x \in B$ and $j$ with $1 \le j \le N(x)$
we have $f^j(x) \not\in B$.
\end{itemize}
\end{subl}

\begin{proof}
Take some positive measure set $C^{(0)}$ of $\Gamma^{(0)} = \Gamma$ where $N$ is constant, say equal to $n_0$.
Since $f$ is aperiodic, we can select a positive measure subset $B^{(0)}$ of $C^{(0)}$
that is disjoint from its first $n_0$ iterates.
Next consider the (invariant) set $\Gamma^{(1)}$ of points in $\Gamma^{(0)}$ whose $f$-orbits never visit $B^{(0)}$.
If $\Gamma^{(1)}$ has zero measure, then we take $B =B^{(0)}$ and we are done.
Otherwise we take a positive measure subset $C^{(1)}$ of $\Gamma^{(1)}$ where $N$ is constant,
and choose $B^{(1)} \subset C^{(1)}$ of positive measure 
that is disjoint from its first $n_1 = N|C^{(1)}$ iterates.
If the set $\Gamma^{(2)}$ formed by the points that never visit $B^{(1)}$
has zero measure then we take $B = B^{(0)} \cup B^{(1)}$ and stop;
otherwise we continue analogously and define $B^{(2)}$ etc.
If this process does not end after finitely many steps then
we define $\Gamma^{(\omega)} = \bigcap_{n < \omega} \Gamma^{(n)}$
and proceed as before, using transfinite induction.
Since a disjoint class of positive measure sets is countable, 
the process will terminate at some countable ordinal.
Taking a union, we find the desired measurable set $B$.
\end{proof}

Let $B$ be given by the sublemma.
For $x \in B$, let $H(x)$ be the minimal positive integer $n$ such that $f^n(x) \in B$.
Then for a.e.\ $x\in B$ we have $N(x) < H(x) < \infty$.

Take $\ell_0 \in \N$ large, and
for $1 \le n \le \ell_0$, take compact sets
$K_n \subset \{x \in B; \; H(x) = n\}$
in a way such that the set
$\Gamma \setminus \bigsqcup_{n=1}^{\ell_0} \bigsqcup_{j=0}^{n-1} f^j (K_n)$
has measure less than $\delta$.
Take open sets $U_n \supset K_n$, all contained in the open set $\Gamma$,
and such that the union $\bigsqcup_{n=1}^{\ell_0} \bigsqcup_{j=0}^{n-1} f^j (U_n)$ is still disjoint.

Let $K = \bigcup_{n=1}^{\ell_0} K_n$.
For each $x \in K$, say with $x \in K_n$,
since $n > N(x)$ we can apply Lemma~\ref{l.lower}
and get a radius $r = r(x)>0$.
If necessary, we reduce $r(x)$  so that $\bar B_{r(x)}(x)$ is contained in the open set $U_n$.
Since $\Phi$ is a measurable $f$-invariant function,
for a.e.~$x$, we can reduce $r(x)$ further and ensure that
\begin{equation}\label{e.density}
\left.
\begin{array}{r}
0 < r < r(x) \\
0 \le j < H(x)
\end{array}
\right\}
\ \Rightarrow \
\frac{1}{\mu (B_r(x))} \mu \big(\{y \in B_r(x) ; \; |\Phi(f^j y) - \Phi(x)| \ge \delta \}\big)
< \frac{\delta}{H(x)} \, .
\end{equation}

Consider the Vitali cover of $K$ by the balls
$\bar B_{r'}(x)$, with $0< r'< r(x)$.
By the Vitali Covering Lemma, there is a countable family of disjoint balls
$\bar B_{r_i}(x_i)$ with $0<r_i < r(x_i)$
that covers the set $K$ mod~$0$.
Write $n_i = H(x_i)$.
By construction, the union
$\bigsqcup_i \bigsqcup_{j=0}^{n_i-1} f^j (\bar B_{r_i}(x_i))$ is still disjoint.

Applying Lemma~\ref{l.lower} for each ball $B_{r_i}(x_i)$,
we get a diffeomorphism $g_i$ close to $f$ such that:
\begin{itemize}
\item $g_i$ equals $f$ outside $\bigsqcup_{j=0}^{n_i-1} f^j(B_{r_i}(x_i))$;
\item there is a set $G_i \subset B_{r_i}(x_i)$ such that
$\mu(G_i) > (1-\delta)\mu(B_{r_i}(x_i))$ and
$$
\frac{1}{n_i} \log \| \mathord{\wedge}^p (Dg_i^{n_i}(y) ) \|  \, \le \,
\Phi(x_i) + \delta \quad
\text{for all $y \in G_i$.}
$$
\end{itemize}
Let us define the global perturbation $g$ of $f$ as follows:
$g$ is equal to $g_i$ in each corresponding $\bigsqcup_{j=0}^{n_i-1} f^j(B_{r_i}(x_i))$,
and equal to $f$ outside.
Then $g$ is a symplectomorphism $C^1$-close to $f$.
We will prove that $g$ has the required properties.

By \eqref{e.density}, for each $j = 0,\ldots, n_i-1$,
$$
\mu \big\{y \in B_{r_i}(x_i) ; \; |\Phi(g^j y) - \Phi(x_i)| \ge \delta \big\}
= \mu \big\{y \in B_{r_i}(x_i) ; \; |\Phi(f^j y) - \Phi(x_i)| \ge \delta \big\}
< \frac{\delta}{n_i} \mu (B_{r_i}(x_i)) \, .
$$
Let $G'_i$ be the set of $y \in G_i$ such that
$|\Phi(g^j (y)) - \Phi(x_i)| < \delta$ for all $j$ with $0 \le j < n_i$.
Then $\mu(G_i') > (1-2\delta)\mu(B_{r_i}(x_i))$.
If $y \in G'_i$ then
\begin{equation} \label{e.sum bound}
\frac{1}{n_i} \log \| \mathord{\wedge}^p (Dg^{n_i}(y) ) \|  \, \le \,
\sum_{j=0}^{n_i-1}\Phi(g^j y) + 2\delta \, .
\end{equation}

Define sets $D_\text{b} = \bigsqcup_i G'_i$ and
$D = \bigsqcup_i \bigsqcup_{j=0}^{n_i-1} g^j(G'_i)$.
(The set $D$ is called a castle with base $D_\text{b}$.)
Let us see that $D$ covers most of $\Gamma$.
Indeed, $\Gamma \setminus D$ is contained mod~$0$ in
$$
\left(\Gamma \setminus \bigsqcup_{n=1}^{\ell_0} \bigsqcup_{j=0}^{n-1} f^j (K_n) \right)
\cup
\left(\bigsqcup_i \bigsqcup_{j=0}^{n_i-1} f^j (B_{r_i}(x_i)) \setminus D \right)
= (\mathrm{I}) \cup (\mathrm{II}) \, .
$$
Recall that $\mu (\mathrm{I}) < \delta$.
On the other hand, $(\mathrm{II}) = \bigsqcup_i \bigsqcup_{j=0}^{n_i-1} g^j(B_{r_i}(x_i) \setminus G'_i)$,
therefore $\mu(\mathrm{II}) < 2\delta$.
This shows that $\mu(\Gamma\setminus D)< 3\delta$.

Let $\hat{D} = \bigcup_{n \ge 0} g^{-n}(D)$.
Almost every $x \in \hat D$ visits $D_\text{b}$ infinitely many times.
Fix one such point $x$, and let
$$
\{m_1 < m_2 < \cdots\} = \{n \ge 0; \; g^n(x) \in D_\text{b} \} \, .
$$
Each $g^{m_j}(x)$ belongs to some ball $B_{r_i}(x_i)$;
consider the corresponding $n_i$ and let $m_j'= m_j + n_i$.
So we have defined numbers $m_1< m_1'\le m_2 < m_2'\le \cdots$
such that $g^n(x)$ is in $D$ if $m_j \le n < m_j'$,
and is not in $D$ if $m_j' \le n < m_{j+1}$.

Given a (large) integer $n$, let $k=k(n)$ be the biggest index such that $m_k' \le n$.
We want to estimate $\|\mathord{\wedge}^p Dg^n (x) \|$;
we start with the following upper bound:
$$
\|\mathord{\wedge}^p Dg^{m_1} (x) \| \;
\|\mathord{\wedge}^p Dg^{m'_1-m_1} (g^{m_1} x) \| \;
\|\mathord{\wedge}^p Dg^{m_2-m'_1} (g^{m_1'} x) \| \cdots
\|\mathord{\wedge}^p Dg^{n-m'_k} (g^{m_k'} x) \| \, .
$$
To estimate some of these factors, we use \eqref{e.sum bound}:
$$
\log \|\mathord{\wedge}^p Dg^{m_j' - m_j} (g^{m_j}x) \|
\le \sum_{i=m_j}^{m_j'-1} \Phi (g^i x) = \sum_{i=m_j}^{m_j'-1} (\Phi \mathbbm{1}_D)(g^i x)
\text{ for each $j=1, \ldots, k-1$.}
$$
To estimate the other factors, let $e^C$ be an upper bound for $\|\mathord{\wedge}^p Dg\|$;
then:
\begin{align*}
\log \|\mathord{\wedge}^p Dg^{m_1} (x) \| &\le Cm_1 \, , \\
\log \|\mathord{\wedge}^p Dg^{m_{j+1} - m_j'} (g^{m_j'}x) \| &\le C(m_{j+1} - m_j')
= \sum_{i=m_j}^{m_j'-1} (C \mathbbm{1}_{\Gamma \setminus D})(g^i x)  \, , \\
\log \|\mathord{\wedge}^p Dg^{n - n'_k} (x) \| &\le C(n - m'_k) \, .
\end{align*}
Putting things together,
\begin{equation}\label{e.almost there}
\log \|\mathord{\wedge}^p (Dg^n (x) )\| \le
C(n - m_k' + m_1) +
\sum_{i=m_1}^{m_k'} \big(C \mathbbm{1}_{\Gamma \setminus D} + \Phi \mathbbm{1}_D\big)(g^i x) \, .
\end{equation}
Now we use:
\begin{subl}
For a.e.~$x \in \hat{D}$, we have that $m_{k(n)}'/n \to 1$ as $n \to \infty$.
\end{subl}

\begin{proof}
It suffices to consider points $x \in D_\text{b}$.
Let $\hat{g}: D_\text{b} \to D_\text{b}$ be the first return map,
and $T: D_\text{b} \to \N$ be return time.
Then $\hat{g}$ preserves $\mu$ restricted to $D_\text{b}$
and $T$ is integrable.
Since $m_{j+1}$ equals the Birkhoff sum $\sum_{i=0}^{j-1} T(\hat g ^i x)$,
for a.e.~$x \in D_\text{b}$,
the $\lim_{j \to \infty} m_j/j$ exists and is positive.
Now, if $k = k(n)$ then $n < m_{k+1}' < m_{k+2}$.
Hence
$$
\frac{m_k}{m_{k+2}} \le  \frac{m'_k}{n} \le 1 \, .
$$
As $n$ goes to infinity, $k=k(n) \to \infty$ and $m_k/m_{k+2} \to 1$.
This proves the sublemma.
\end{proof}

It follows from \eqref{e.almost there} and the sublemma that for a.e.~$x \in \hat{D}$,
$$
\Lambda_p(g,x) = \lim_{n \to \infty} \frac 1n \log \|\mathord{\wedge}^p (Dg^n (x) )\| \le
\lim_{n \to \infty} \frac 1n \sum_{i=0}^{n-1} \big(C \mathbbm{1}_{\Gamma \setminus D} + \Phi \mathbbm{1}_D\big)(g^i x) \, .
$$
The same is obviously true if $x\in \Gamma \setminus \hat D$.
Integrating over $x \in \Gamma$, we obtain
$$
\int_{\Gamma} \Lambda_p(g) \le \int (C \mathbbm{1}_{\Gamma \setminus D} + \Phi \mathbbm{1}_D)
\le  3C\delta + \int_D \Phi
\le  3C\delta + \int_\Gamma \Phi \, .
$$
This gives \eqref{e.global} (replace $\delta$ with $\delta/(3C)$ everywhere),
and therefore the proposition is proved.
\end{proof}

\begin{rem}
With some additional work one can show that the aperiodicity hypothesis is not necessary
for the validity of Proposition~\ref{p.jump};
indeed it does not appear in \cite{BV Annals}.
\end{rem}

Now it is done:

\begin{proof}[Proof of Theorem~\ref{t.continuity}]
Let $\cA$ be the residual subset of $\Diff^1_\omega(M)$ formed by aperiodic diffeomorphisms.
Consider the semi-continuous maps $\LE_p: \cA \to \R$.
Since $\cA$ is also a Baire space,
it follows that there is a residual subset $\cR$ of $\cA$
(and hence also residual as a subset $\Diff^1_\omega(M)$)
such that every $f \in \cR$ is a point of continuity of  each $\LE_p$, with $p = 1, \ldots, N$.
Fix one such $f$; by Proposition~\ref{p.jump}, each $J_p(f)$ vanishes.
This implies that for almost every regular point $x\in M$,
if $p \le N$ is such that $\lambda_p(f,x) > \lambda_{p+1}(f,x)$
then $x$ does not belong to $\Gamma(f,\infty)$.
That is, there is a dominated splitting $E^u \oplus F$ of index $p$ along the orbit of $x$.
Theorem~\ref{t.DS is PH} implies that $E^u \oplus F$
can be refined to a partially hyperbolic splitting $E^u \oplus E^c \oplus E^s$, with
$\dim E^s = \dim E^u = p$.
Thus $E^u$, $E^c$, and $E^s$ must be the sum of the Oseledets spaces
associated to the Lyapunov exponents $\lambda_i(f,x)$
respectively with
$$
1\le i \le p, \quad  p < i \le 2N-p, \quad  2N-p < i \le 2N.
$$
All this holds whenever $\lambda_p(f,x) > \lambda_{p+1}(f,x)$,
so proving that the Oseledets splitting is dominated along the orbit of $x$.
\end{proof}

Theorem~\ref{t.main} is an immediate consequence of Theorem~\ref{t.continuity}.

\section[Results for partially hyperbolic maps]{Results for Partially Hyperbolic Maps} \label{s.PH proof}

We will obtain Theorem~\ref{t.global PH} as a corollary of the slightly more technical
Theorem~\ref{t.finest} below.

First of all, we need the following two results about
the well-known accessibility property from partially hyperbolic theory:

\begin{othertheorem}[Dolgopyat and Wilkinson \cite{Dolgo Wilk}] \label{t.DW}
There is an open and dense set $\cA \subset \PH^1_\omega (M)$
formed by accessible symplectomorphisms.
\end{othertheorem}

\begin{othertheorem}[Brin~\cite{Brin}]\label{t.Brin}
If $f$ is a $C^2$ volume-preserving partially hyperbolic diffeomorphism
with the accessibility property then almost every point has a dense orbit.
\end{othertheorem}

In fact, Brin proved the result for \emph{absolute} partially hyperbolic maps (recall Remark~\ref{r.definitions PH}).
Another proof was given by Burns, Dolgopyat, and Pesin, see \cite[Lemma 5]{Burns Dolgo Pesin}
(or \cite[\S7.2]{Hass Pesin}).
Their proof also applies to \emph{relative} partially hyperbolic maps:
the only necessary modification is to use
the property of absolute continuity of stable and unstable foliations in the relative case,
which is proven by Abdenur and Viana in \cite{AV flavors}.

In order to extract from Theorem~\ref{t.Brin} consequences for $C^1$ maps,
we need the following well-known result:

\begin{othertheorem}[Zehnder~\cite{Zehnder}] \label{t.Z}
$C^\infty$ diffeomorphisms form a dense subset of $\Diff^1_\omega(M)$.
\end{othertheorem}
We remark that the volume-preserving analogue of Theorem~\ref{t.Z}
was recently obtained by Avila~\cite{Avila smoothening}.

As a consequence of the above theorems, we obtain:

\begin{prop}\label{p.generic Brin}
For a generic $f$ in $\PH^1_\omega (M)$,
the orbit of almost every point is dense in $M$.
\end{prop}

\begin{proof}
Given $f \in \PH_\omega^1(M)$, let $D(f)$ be the set of points in $M$ whose $f$-orbits are dense.
Let $\cR$ be set of $f \in \PH_\omega^1(M)$ such that $m(D(f))=1$.
Theorems~\ref{t.Z}, \ref{t.DW}, and \ref{t.Brin} together imply that $\cR$ is dense in $\PH_\omega^1(M)$.
We will complete the proof showing that $\cR$ is a $G_\delta$ set.

Let $\cB$ be a countable basis of (non-empty) open sets of $M$.
Then
$$
D(f) = \bigcap_{U,V\in \cB} G(U,V,f) \, \quad \text{where} \quad
G(U,V,f) = (M\setminus U) \cup \bigcup_{n\in \N} f^{-n}(V) \, .
$$
For $k \in \N$, let
$\cA(U,V,k)$ be the set of $f \in \PH_\omega^1(M)$
such that $m(G(U,V,f)) > 1 -1/k$.
Then each $\cA(U,V,k)$ is open.
Their intersection is precisely the set of $f\in \PH^1_\omega(M)$ such that
$m(D(f))=1$, that is, $\cR$.
\end{proof}

A dominated splitting $TM= E^1 \oplus \cdots \oplus E^k$ (into non-zero bundles)
for a  diffeomorphism $f \colon M \to M$
is called the \emph{finest dominated splitting}
if there is no dominated splitting defined over all $M$ with more than $k$ (non-zero) bundles.
For any $f$, either
there is no dominated splitting over $M$,
or there is a unique finest dominated splitting (and moreover it refines every dominated splitting on $M$).
See~\cite{BDV livro}.

Now we can state and prove the:

\begin{thm}\label{t.finest}
For a generic $f$ in $\PH^1_\omega(M)$,
the Oseledets splitting at almost every point
coincides with the finest dominated splitting of $f$.
In particular, the multiplicities of the Lyapunov exponents are a.e.\ constant.
\end{thm}

\begin{proof}
Let $k(f)$ denote the number of bundles in the finest dominated splitting of a map ${f \colon M \to M}$.
Then the Oseledets splitting at any regular point for $f$ has at least $k(f)$ bundles.
Now let $f\in \PH_\omega^1(M)$
satisfy the generic properties from Proposition~\ref{p.generic Brin} and Theorem~\ref{t.main}.
That is, for almost every $x \in M$, the orbit of $x$ is dense
and the Oseledets splitting along it is (non-trivial and) dominated.
The Oseledets splitting along the orbit of any such point
extends to a dominated splitting over $M$, and hence must have exactly $k(f)$ bundles.
\end{proof}

As a consequence:

\begin{proof}[Proof of  Theorem~\ref{t.global PH}]
If $f$ belongs to the residual set given by Theorem~\ref{t.finest}
then the Oseledets space corresponding to zero exponents (if they exist)
coincides a.e.\ with the ``middle'' bundle of the finest dominated splitting,
which by Theorem~\ref{t.DS is PH} is
the center bundle of a partially hyperbolic splitting.
\end{proof}


\bigskip

\begin{ack}
I thank Artur Avila for pointing that the earlier version of the proof of Proposition~\ref{p.jump} 
could be significantly simplified,
and the referee for his/her careful reading.
This work started during my visit to
the Morningside Center of Mathematics in Beijing.
My sincere thanks for the hospitality.
\end{ack}


\vfill

\noindent PUC (Rio de Janeiro).


\noindent www.mat.puc-rio.br/$\sim$jairo

\noindent jairo@mat.puc-rio.br

\end{document}